\documentclass[11pt]{amsart}
\usepackage[foot]{amsaddr}
\usepackage{amsfonts}
\usepackage[latin9]{inputenc}
\usepackage{amsmath}
\usepackage{amssymb}
\usepackage{breqn}
\usepackage{url}
\usepackage{graphicx}
\usepackage[pdfstartview=FitH]{hyperref}
\usepackage{amsthm}
\usepackage{hyperref}
\usepackage{url}
\usepackage{cleveref}
\usepackage{verbatim}
\usepackage{bbm}

\calclayout

\crefformat{section}{\S#2#1#3} 
\crefformat{subsection}{\S#2#1#3}
\crefformat{subsubsection}{\S#2#1#3}

\newcommand{\proj}{\operatorname{proj}}

\newcommand{\im}{\operatorname{Im}}

\newcommand{\B}{\operatorname{\mathbb{E}}}

\newcommand{\supp}{\operatorname{supp}}
\newcommand{\St}{\operatorname{St}}
\newcommand{\SL}{\operatorname{SL}}
\newcommand{\GL}{\operatorname{GL}}

\newcommand{\UT}{\operatorname{UT}}
\newcommand{\LT}{\operatorname{LT}}
\newcommand{\Prob}{\operatorname{Prob}}
\newcommand{\Isom}{\operatorname{Isom}}

\makeatletter
\renewcommand{\email}[2][]{%
  \ifx\emails\@empty\relax\else{\g@addto@macro\emails{,\space}}\fi%
  \@ifnotempty{#1}{\g@addto@macro\emails{\textrm{(#1)}\space}}%
  \g@addto@macro\emails{#2}%
}
\makeatother

\newcommand{\RN}[1]{%
  \textup{\uppercase\expandafter{\romannumeral#1}}%
}

\makeatletter
\begin{document}
\newtheorem{theorem}{Theorem}[section]
\newtheorem{lemma}[theorem]{Lemma}
\newtheorem{claim}[theorem]{Claim}
\newtheorem{proposition}[theorem]{Proposition}
\newtheorem{corollary}[theorem]{Corollary}
\theoremstyle{definition}
\newtheorem{definition}[theorem]{Definition}
\newtheorem{observation}[theorem]{Observation}
\newtheorem{example}[theorem]{Example}
\newtheorem{remark}[theorem]{Remark}

\title[]{Banach property (T) for $\rm SL_n (\mathbb{Z})$ and its applications}
\author{Izhar Oppenheim}
\address{Department of Mathematics, Ben-Gurion University of the Negev, Be'er Sheva 84105, Israel}
\email{izharo@bgu.ac.il}
\thanks{The author was partially supported by ISF grant no. 293/18}
\maketitle
\begin{abstract}
We prove that  a large family of higher rank simple Lie groups (including $\rm SL_n (\mathbb{R})$ for $n \geq 3$) and their lattices have Banach property (T) with respect to all super-reflexive Banach spaces. 

Two consequences of this result are: First,  we deduce Banach fixed point properties with respect to all super-reflexive Banach spaces for a large family of higher rank simple Lie groups.  For example,  we show that for every $n \geq 4$,  the group $\rm SL_n (\mathbb{R})$ and all its lattices have the Banach fixed point property with respect to all super-reflexive Banach spaces.  Second,  we settle a long standing open problem and show that the Margulis expanders (Cayley graphs of $\rm SL_{n} (\mathbb{Z} / m \mathbb{Z} )$ for a fixed $n \geq 3$ and $m$ tending to infinity) are super-expanders.  

All of our results stem from proving Banach property (T) for $\rm SL_3 (\mathbb{Z})$.  Our method of proof for $\rm SL_3 (\mathbb{Z})$ relies on a novel proof for relative Banach property (T) for the uni-triangular subgroup of $\SL_3 (\mathbb{Z})$.  This proof of relative property (T) is new even in the classical Hilbert setting and is interesting in its own right.  
\end{abstract}


\section{Introduction}

Property (T) was introduced by Kazhdan in \cite{Kazhdan} as a tool to prove finite generation.  Since then it was found useful for a wide range of applications in various different areas of mathematics (see \cite{PropTBook} and the introduction of \cite{BFGM},  and reference therein).  We mention two such applications that are relevant in the context of this paper: First, property (T) for a group $G$ is equivalent (under some mild assumptions on $G$) to property (FH) which states that every continuous isometric affine action of $G$ on a real Hilbert space admits a fixed point.  Second,  Margulis gave the first explicit construction of expanders using property (T).

More recently,  Bader, Furman, Gelander and Monod \cite{BFGM} defined a Banach version of property (T) (and its connection to Banach fixed point properties).   They conjectured that higher rank algebraic group should have this form of Banach property (T) with respect to the class of all super-reflexive Banach spaces.  Roughly simultaneously to the work of \cite{BFGM}, V. Lafforgue \cite{Laff2} proved that groups of the form $\SL_3 (F)$ where $F$ is a non-Archimedian local field have a strong form of Banach property (T) for large classes of Banach spaces and that this strong form of Banach property (T) implies the fix point property.  In particular,  his work corroborates the conjecture of \cite{BFGM}: Namely,  a consequence of Lafforgue's work it that the groups $\SL_3 (F)$ where $F$ is a non-Archimedian local field have Banach property (T) and the fixed point property with respect to all super-reflexive Banach spaces.  Later,  Liao \cite{Liao} extended the work of Lafforgue and proved the strong version of Banach property (T) holds for every higher rank connected almost $F$-simple algebraic group,  where $F$ is a non-Archimedean local field.  In his work \cite{Laff2},  Lafforgue also showed how to use his result to construct super-expanders, i.e., families of graphs that are expanders with respect to every super-reflexive Banach space (see exact definition below). 

Unlike in the non-Archimedian case,  much less was known regarding Banach property (T) for algebraic groups over $\mathbb{R}$ (and their lattices).  In the paper of Bader,  Furman, Gelander and Monod \cite{BFGM} they showed that higher rank algebraic groups have Banach property (T) (and fixed point properties) for $L^p$ spaces where $1 < p < \infty$ (the case $p=1$ was later resolved in \cite{BGM}).  For general super-reflexive spaces (that are not $L^p$ spaces),  partial results were proven by de Laat, Mimura and de la Salle in various collaborations \cite{LaatSalle3, Salle,LMS,  LaatSalle2, Salle2}.  However,  non of these works cover all super-reflexive spaces even for a single group of the form $\SL_n (\mathbb{R})$ with some $n \geq 3$.   

In this paper,  we make a major breakthrough regarding Banach property (T):  We show that large family of connected simple (higher rank) Lie groups have Banach property (T) with respect to all super-reflexive Banach spaces.  In particular,  we show that for every $n \geq 3$,  the groups $\SL_n (\mathbb{Z})$ and $\SL_n (\mathbb{R})$ have Banach property (T) with respect to all super-reflexive Banach spaces.  This has several striking consequences: First, it allows us to prove that a family of connected simple (higher rank) Lie groups and their lattices  have the fixed point property with respect to all super-reflexive Banach spaces.  In particular,  we show that for every $n \geq 4$,  $\SL_n (\mathbb{R})$ and all its lattices have the fixed point property with respect to all super-reflexive Banach spaces.  Second,  we settle a long standing open problem and show that the Margulis expanders (i.e.,  Cayley graphs of $\SL_n (\mathbb{Z} / i \mathbb{Z})$ where $n \geq 3$ is a fixed integer) are super-expanders.  Last,  we show that for every  $n \geq 5$,  the groups $\SL_n (\mathbb{Z})$ and $\SL_n (\mathbb{R})$ have a strengthening of the fixed point property (property $(FF)$ defined below) with respect to all super-reflexive Banach spaces

Our method for proving Banach property (T) is also novel.  The prior works of de Laat, Mimura and de la Salle mentioned above were based on generalizing the work of Lafforgue on strong (Hilbert) property (T) of $\SL_3 (\mathbb{R})$ to the Banach setting.  Our approach is very different: We first prove a relative version of Banach property (T) for the uni-triangular matrices in $\SL_3 (\mathbb{Z})$ with respect to super-reflexive Banach spaces.  We note that this proof is new even in the Hilbert setting.  After that,  we use a bounded generation argument \`{a} la Shalom to deduce Banach property (T) for $\SL_3 (\mathbb{Z})$ for all super-reflexive spaces.  Then using Howe-Moore we deduce Banach property $(T)$ for simple Lie groups with whose Lie algebra contain $\mathfrak{sl}_3 (\mathbb{R})$ (see exact formulation below).  We note that most of this proof scheme was known to experts, but prior to this work, it was not known how to prove the relative Banach property $(T)$ result.

\subsection{Uniformly convex and super-reflexive Banach spaces}

A Banach space $\B$ is called \textit{uniformly convex} if there is a function $\delta: (0,2] \rightarrow (0,1]$ called the \textit{modulus of convexity} such that for every $0< \varepsilon \leq 2$ and every $\xi,  \eta \in \B$ with $\Vert \xi \Vert = \Vert \eta \Vert =1$, if $\Vert \xi -\eta \Vert \geq \varepsilon$, then $\Vert \frac{\xi+ \eta}{2} \Vert \leq (1- \delta (\varepsilon))$.   

We will not recall the definition of super-reflexive Banach spaces, but only note that by \cite[Theorem A.6]{BL} a Banach space $\B$ is super-reflexive if and only if there is a equivalent uniformly convex norm on $\B$ (a reader who is not familiar with super-reflexive Banach spaces can take this as a definition).  


\subsection{Banach property (T) for $\SL_3 (\mathbb{Z})$ and Simple Lie groups} 

Given a topological group $G$ and a Banach space $\B$,  a \textit{linear isometric representation of $G$ on $\B$} is a continuous homomorphism $\pi : G \rightarrow O (\B)$, where $O (\B)$ denotes the group of all invertible linear isometries of $\B$ with the strong operator topology.  A linear isometric representation $\pi$ is said to have \textit{almost invariant vectors} if for every compact set $K \subseteq G$ and every $\varepsilon >0$,  there is a unit vector $\xi \in \B$ such that 
$$\sup_{g \in K} \Vert  \pi  (g) \xi - \xi \Vert < \varepsilon.$$

In \cite{BFGM},  Bader,  Furman,  Gelander and Monod defined Banach property (T) of a group $G$ as follows:
\begin{definition}
\label{Banach property T def}
Let $\B$ be Banach space and $G$ be a topological group.  The group $G$ has property $(T_\B)$ if for every continuous linear isometric representation $\pi : G  \rightarrow O( \B)$,  the quotient representation $\pi ' : G \rightarrow O (\B / \B^{\pi (G)})$ does not have almost invariant vectors.  
\end{definition}

The main result of this paper is proving Banach property $(T_\B)$ with respect to every super-reflexive Banach space $\B$ for $\SL_3 (\mathbb{Z})$:
\begin{theorem}
\label{T for SL_3 Z intro thm}
Let $\B$ be a super-reflexive Banach space.  The group  $\SL_3 (\mathbb{Z})$ has property $(T_{\B})$.  
\end{theorem}

Using Howe-Moore Theorem,  this allows us to deduce the following theorem:
\begin{theorem}
\label{H rank intro thm}
Let $G$ be a connected simple real Lie group with a finite center and $\mathfrak{g}$ the Lie algebra of $G$.  If  $\mathfrak{g}$ contains $\mathfrak{sl}_3 (\mathbb{R})$ as a Lie sub-algebra, then $G$ and all its lattices have property $(T_{\B})$ for every super-reflexive Banach space $\B$.  
\end{theorem}

This corroborates a conjecture stated in \cite{BFGM} in which it was conjectured that all higher rank almost simple algebraic groups have property $(T_\B)$ for every super-reflexive Banach space $\B$ (see \cite[Remark 2.28]{BFGM}).  An immediate consequence is:
\begin{theorem}
\label{T for SL_n R intro thm}
Let $n \geq 3$ and $\B$ be a super-reflexive Banach space.  The group  $\SL_n (\mathbb{R})$ and all its lattices have property $(T_{\B})$.  In particular,  for every $n \geq 3$ and every super-reflexive Banach space $\B$,  the group  $\SL_n (\mathbb{Z})$ has property  $(T_{\B})$.
\end{theorem}

\subsection{Applications}

\paragraph{\textbf{Banach fixed point property for simple Lie groups}}

Given a Banach space $\B$,  a topological group $G$ is said to have property $(F_{\B})$ is every affine (continuous) isometric action of $G$ on $\B$ admits a fixed point.  

As an application of Banach property $(T)$ for $\SL_{3} (\mathbb{R})$ we prove:
\begin{theorem}
\label{H rank intro thm 2}
Let $G$ be a connected simple real Lie group with a finite center and $\mathfrak{g}$ be the Lie algebra of $G$.  If $\mathfrak{g}$ contains $\mathfrak{sl}_4 (\mathbb{R})$ as a Lie sub-algebra, then $G$ and any lattice $\Gamma < G$ have property $(F_\B)$ for every super-reflexive Banach space $\B$.  
\end{theorem}

Again, this corroborates a conjecture stated in \cite{BFGM} in which it was conjectured that all higher rank algebraic groups have property $(F_\B)$ for every super-reflexive Banach space $\B$ (see \cite[Conjecture 1.6]{BFGM}).  

The method of proof of Theorem \ref{H rank intro thm 2} is as follows: First,  using ideas of \cite{BFGM},  we show that Banach property $(T_\B)$ for $\SL_{3} (\mathbb{R})$ implies property $(F_\B)$ for $\SL_{4} (\mathbb{R})$.  Second,  via a (now standard) use of Maunter phenomenon,  we deduce Theorem \ref{H rank intro thm 2}. \\

\paragraph{\textbf{Super-expanders}}

A family of finite graphs with uniformly bounded degree is called a \textit{super-expander family} (or a \textit{super-expander}) if it has a Poincar\'{e} inequality with respect to every super-reflexive Banach space (see exact definition in \cref{Super-expanders subsec}). The first examples of super-expanders were constructed by Lafforgue in \cite{Laff2}  as a consequence of his work on strong Banach property (T) for $\SL_3 (F)$, where $F$ is a non-Archimedean local field.  Since Lafforgue's work there have been several constructions of super-expanders using two main techniques: Namely,  the work of Mendel and Naor on non-linear spectral calculus \cite{MendelNaor} which gave a zig-zag construction for super-expanders.  In a different direction,  several works \cite{warpedcones0, warpedcones1, warpedcones2, warpedcones3} gave constructions using warped cones of groups actions arising from groups with Banach property (T).  

It was an open problem to determine whether the Margulis expanders is a super-expander family,  i.e.,  if for a fixed $n \geq 3$ the Cayley graphs of $\SL_n (\mathbb{Z} / i \mathbb{Z})$ form a super-expander family.  This open question appeared in the literature several times,  including in Assaf Naor's Minerva lecture \cite{Minerva} where it was attributed to Margulis and in de la Salle's 2022 ICM lecture \cite[Conjecture 4.4]{dlSICM} (see also  \cite[Problem 5]{Oberw}, \cite{LaatSalle2}, \cite[Remark 5.3]{Mimura2}).  Partial results were achieved by de Laat and de la Salle in \cite{LaatSalle2}, but up until our work, the problem remained open.  As a consequence of our Theorem \ref{T for SL_n R intro thm}, we settle this problem to the affirmative and prove the following:
\begin{theorem}
\label{super-exp intro thm}
Let $n \geq 3$ and $S$ be a finite generating set of $\SL_n (\mathbb{Z})$ (e.g., $S = \lbrace e_{i,j} (\pm 1) : 1 \leq i, j \leq n, i  \neq j \rbrace$).  Also, let $\phi_i : \SL_n (\mathbb{Z}) \rightarrow \SL_n (\mathbb{Z} / i \mathbb{Z})$ be the natural surjective homomorphism for every $i \in \mathbb{N}$.  Then the family of Cayley graphs of $\lbrace (\SL_{n} (\mathbb{Z} / i \mathbb{Z} ),  \phi_i (S)) \rbrace_{i \in \mathbb{N}}$ is a super-expander family.
\end{theorem}

As noted above,  one can construct super-expanders using warped cones arising from an action of a Banach property (T) group on a compact Riemannian manifold (see  \cite{warpedcones0, warpedcones1, warpedcones2, warpedcones3}).  Combining this machinery with our Theorem \ref{T for SL_n R intro thm} also leads to a construction of super-expanders (see \cref{Super-expanders subsec} for more details):
\begin{theorem}
\label{wraped cones super-exp intro thm}
Let $n \geq 3$ and let $M$ be a compact Riemannian manifold such that $\SL_n (\mathbb{Z})$ acts on $M$ by Lipschitz homeomorphisms.  
For every increasing sequence $\lbrace t_i \rbrace_{i \in \mathbb{N}} \subseteq \mathbb{R}_{>0}$ tending to infinity,   the family $\lbrace (M,  d^{t_i}_{\SL_n (\mathbb{Z})}) \rbrace_{i \in \mathbb{N}}$ is quasi-isometric to a super-expander.
\end{theorem}

\paragraph{\textbf{Banach property $(FF_\B)$}}

In \cite{Mimura1},  Mimura defined the notion of property $(FF_{\B})$ (see exact definition in  \cref{prop FFB subsec} below) that is a Banach version of property (TT) defined by Monod in \cite{Monod}.  A result of de Laat,  Mimura and de la Salle \cite{LMS} allows one to deduce property $(FF_{\B})$ for the groups $\SL_{n+2} (\mathbb{Z}),  \SL_{n+2} (\mathbb{R})$  from property $(T_{\B})$ for the groups $\SL_n (\mathbb{Z}), \SL_n (\mathbb{R})$.  Thus, we can deduce the following:
\begin{corollary}
\label{FF intro coro}
For every $n \geq 5$ and every super-reflexive Banach space $\B$,  the groups $\SL_n (\mathbb{Z}),  \SL_n (\mathbb{R})$ have property $(FF_{\B})$.
\end{corollary}

\subsection{Relative Banach property (T) for uni-triangular in $\SL_3 (\mathbb{Z})$} 

Here we will outline the proof of Theorem \ref{T for SL_3 Z intro thm} from which all our other results follow.  The proof relies on a relative Banach property (T) argument that is novel even in the Hilbert setting.    

Generalizing the definition of relative property (T) given in \cite[Theorem 1.2 (b2)]{Jol} to the Banach setting, we will define relative Banach property (T) as follows:
\begin{definition}
\label{relative property T Jol def}
Let $G$ be a topological group and $H<G$ be a subgroup.  For a Banach space $\B$,  we will say that the pair $(G,H)$ has relative Banach property $(T_\B)$ if for every continuous linear isometric representation $\pi : G \rightarrow O (\B)$ and every constant $\gamma >0$,  there is a compact set $K \subseteq G$ and a constant $\varepsilon>0$ such that for every unit vector $\xi \in \B$, if $\sup_{g \in K} \Vert  \pi  (g) \xi - \xi \Vert < \varepsilon$, then there exits $\eta \in \B^{\pi (H)}$ such that $ \Vert  \xi - \eta \Vert < \gamma$. 
\end{definition}

\begin{remark}
This definition is a-priori weaker than definition of weak relative Banach property (T) given in \cite{LMS} and is strictly weaker than the definition of strong relative Banach property (T) given in \cite{BFGM}.
\end{remark}

Let $\UT_3 (\mathbb{Z})$ and $\LT_3 (\mathbb{Z})$ denote the subgroups of uni-upper-triangular and uni-lower-triangular matrices in  $\SL_3 (\mathbb{Z})$, i.e., 
$$\UT_3 (\mathbb{Z}) = \left\lbrace \left( \begin{matrix}
1 & a & c \\
0 & 1 & b \\
0& 0 &1 
 \end{matrix} \right) : a,b,c \in \mathbb{Z} \right\rbrace,$$
 and
 $$\LT_3 (\mathbb{Z}) = 
\left\lbrace
\left( \begin{matrix}
1 & 0 & 0 \\
a & 1 & 0 \\
c & b & 1
\end{matrix} \right) : a,b, c \in \mathbb{Z} \right\rbrace.$$
 
We prove the following (with respect to our definition of relative Banach property $(T)$ stated above):
\begin{theorem}
\label{relative prop T - intro thm}
For any super-reflexive Banach space $\B$ the pairs $(\SL_3 (\mathbb{Z}), \UT_3 (\mathbb{Z}))$ and $(\SL_3 (\mathbb{Z}), \LT_3 (\mathbb{Z}))$ both have relative property $(T_\B)$. 
\end{theorem}
The proof of this theorem is via defining a sequence of finitely supported probability measures $f_n$ on $\SL_3 (\mathbb{Z})$ and showing that for every super-reflexive Banach space $\B$ and every continuous isometric representation $\pi : \SL_3 (\mathbb{Z}) \rightarrow O (\B)$ it holds that $\pi (f_n)$ converges in the norm topology and that the image of the limit operator is in $\B^{\pi (\UT_3 (\mathbb{Z}))}$.  We note that this method is new even in the classical Hilbert setting.  Moreover, the proof is completely elementary in contrast with the more classical proofs of relative property (T), e.g.,  the proof that $(\SL_2 (\mathbb{Z}) \ltimes \mathbb{Z}^2,  \mathbb{Z}^2)$ has relative property (T) in \cite[Theorem 4.2.2]{PropTBook} requires Fourier analysis and projections valued measures while our proof requires neither.

Combining Theorem \ref{relative prop T - intro thm} with a bounded generation argument \`{a} la Shalom implies Theorem \ref{T for SL_3 Z intro thm}.  \\

\paragraph{\textbf{Structure of this paper}}
This paper is organized as follows: In \cref{prelim sec}, we cover some needed preliminaries. In \cref{Banach prop T sec},  we gather some facts regarding  Banach property (T) (and relative versions of Banach property (T)).  In \cref{Bounded generation and Banach property (T) sec},  we show how bounded generation and relative Banach property (T) imply Banach property (T).  In \cref{Averaging operations sec},  we prove some bounds on the norms of averaging operations for the Heisenberg group that are needed for our relative Banach property (T) result.  In \cref{Relative Banach property sec},  we prove our relative Banach property (T) result (Theorem \ref{relative prop T - intro thm}).  In \cref{SL sec}, we prove Banach property (T) for $\SL_{3} (\mathbb{Z})$ (Theorem \ref{T for SL_3 Z intro thm}).   In \cref{Banach property $(T)$ for simple Lie groups},  we prove Banach property (T) for a large family of simple Lie groups and their lattices (Theorem \ref{H rank intro thm}).  Last,  in \cref{app sec},  we prove the applications stated above.  \\

\paragraph{\textbf{Acknowledgements}} 
I thank Uri Bader and Mikael de la Salle for reading an early draft of this paper and making several valuable suggestions that vastly improved it.  I also thank the anonymous referee of this paper for many corrections and suggestions that improved the correctness and readability of this paper.

\section{Preliminaries}

\label{prelim sec}

\subsection{Uniformly convexity}
\label{Uniformly convexity subsec}

Below,  we will state some needed facts regarding uniformly convex spaces.  

\begin{proposition}
\label{uc ineq prop}
Let $\B$ be a uniformly convex Banach space with a modulus of convexity $\delta :  (0,2] \rightarrow (0,1]$ and denote $O(\B)$ to be the group of invertible linear isometries of $\B$.  Then for every $0 < \varepsilon \leq 2$, every $S,T  \in O(\B)$ such that $TS = ST$ and every $\xi \in \B$,  if $\Vert (I-S) \xi \Vert \geq \varepsilon \Vert \xi \Vert$,  then
$$\frac{1}{2} \left\Vert \frac{I+T}{2}  \xi \right\Vert + \frac{1}{2} \left\Vert \frac{I+TS}{2} \xi \right\Vert \leq \max \left\lbrace  1-\frac{1}{2} \delta (\delta (\varepsilon)),  1- \frac{1}{4} \delta (\varepsilon) \right\rbrace \Vert \xi \Vert.$$

\end{proposition}

This proposition is probably well-known and we give the proof for completeness:

\begin{proof}


Fix $\xi \in \B$ and $0 < \varepsilon \leq 2$,  and assume that $\Vert (I-S) \xi \Vert \geq \varepsilon \Vert \xi \Vert$.  If $\Vert (I-T) \xi \Vert \geq \delta (\varepsilon)$, then 
\begin{align*}
\frac{1}{2} \left\Vert \frac{I+T}{2}  \xi \right\Vert + \frac{1}{2} \left\Vert \frac{I+TS}{2} \xi \right\Vert \leq 
\frac{1}{2} \left(1- \delta (\delta (\varepsilon)) \right) \Vert \xi \Vert + \frac{1}{2} \Vert \xi \Vert = \left(1-\frac{1}{2} \delta (\delta (\varepsilon)) \right) \Vert \xi \Vert,
\end{align*}
as needed. 

Otherwise,  $\Vert \frac{I-T}{2} \xi \Vert \leq \frac{\delta (\varepsilon)}{2}$ and 
\begin{align*}
\frac{1}{2} \left\Vert \frac{I+T}{2}  \xi \right\Vert + \frac{1}{2} \left\Vert \frac{I+TS}{2} \xi \right\Vert \leq 
\frac{1}{2}  \Vert \xi \Vert + \frac{1}{2} \left\Vert \frac{I+S}{2} \xi \right\Vert +  \frac{1}{2} \left\Vert \frac{S(T-I)}{2} \xi \right\Vert \leq \\
\frac{1}{2}  \Vert \xi \Vert + \frac{1}{2} (1- \delta (\varepsilon)) \Vert \xi \Vert + \frac{1}{4} \delta (\varepsilon) \Vert \xi \Vert = \left(1- \frac{1}{4} \delta (\varepsilon) \right) \Vert \xi \Vert.
\end{align*}

\end{proof}

We will be interested in classes of uniformly convex Banach spaces defined as follows: Let $\delta_0 : (0,2] \rightarrow (0,1]$ be a monotone increasing function. Denote $\mathcal{E}_{us} (\delta_0)$ to be the class of all uniformly convex Banach spaces $\B$ with such that the  modulus of convexity of $\B$ is bounded by $\delta_0$, i.e., for a uniformly convex Banach space $\B$ with a modulus of convexity $\delta : (0,2] \rightarrow (0,1]$ it holds that $\B \in \mathcal{E}_{uc} (\delta_0)$ if and only if for every $0<\varepsilon \leq 2$ it holds that $ \delta (\varepsilon)\geq \delta_0 (\varepsilon)$.   For these classes of Banach space, we state the following immediate corollary of Proposition \ref{uc ineq prop}:

\begin{corollary}
\label{uc ineq coro}
Let $\delta_0 : (0,2] \rightarrow (0,1]$ be a function and $0< \varepsilon \leq 2$ be a constant. There is $r_0 = r_0 (\delta_0, \varepsilon)$,  $0 \leq r_0 < 1$ such that for every $\B \in \mathcal{E}_{uc} (\delta_0)$ and every two commuting operators $S,T \in O (\B)$ it holds for every $\xi \in \B$ that if $\Vert (I-S) \xi \Vert \geq \varepsilon \Vert \xi \Vert$,  then
$$\frac{1}{2} \left\Vert \frac{I+T}{2}  \xi \right\Vert + \frac{1}{2} \left\Vert \frac{I+TS}{2} \xi \right\Vert \leq r_0 \Vert \xi \Vert.$$
\end{corollary}

We will need the following theorems:
\begin{theorem}\cite[Theorem 3]{Day}
\label{l^2 sum of uc spaces thm}
Let $\delta_0 : (0,2] \rightarrow (0,1]$ be a function and $\lbrace \B_n \rbrace_{n \in \mathbb{N}} \subseteq \mathcal{E}_{uc} (\delta_0)$ a sequence.  Then the $\ell^2$-sum $\bigoplus_{n} \B_n$ is a uniformly convex Banach space.
\end{theorem}

\begin{theorem}
\label{L^2 sum of uc spaces thm}
\cite[Theorem 1.e.9]{BSpacesBook}
Let $\delta_0 : (0,2] \rightarrow (0,1]$ be a function. There is a function $\delta_0 ' :  (0,2] \rightarrow (0,1]$
 that for every finite measure space $(X, \mu)$ and every $\B \in  \mathcal{E}_{uc} (\delta_0)$ it holds that $L^2 (X, \mu ; \B) \in \mathcal{E}_{uc} (\delta_0 ')$.  
\end{theorem}

\begin{remark}
In the above theorem,  we implicitly use the fact that $L^2 (X, \mu ; \B^*)$ is isometrically isomorphic to $(L^2 (X, \mu ; \B))^*$.  This follows from \cite[Section IV.1,  Theorem 1]{DU} combined with the fact that reflexive Banach spaces (and in particular uniformly convex Banach spaces) have the Radon-Nikod\'{y}m property (see  \cite[Section III.2,  Corollary 13]{DU}).
\end{remark}

\subsection{Linear Representation of groups on Banach spaces}

Given a topological group $G$ and a Banach space $\B$,  a \textit{linear representation of $G$ on $\B$} is a continuous homomorphism $\pi : G \rightarrow \GL (\B)$, where $\GL (\B)$ denotes the group of all invertible linear transformations of $\B$ with the strong operator topology.   Throughout this paper,  $\pi$ will  denote a continuous representation.  

For a given linear representation of $G$ on $\B$, the contragredient representation is the map $\pi^* : G \rightarrow \GL (\B^*)$ defined as 
$$\forall g \in G,  \forall \xi \in \B, \eta \in \B^*,  \langle \pi (g)\xi, \eta \rangle =  \langle \xi, \pi^* (g^{-1}) \eta \rangle .$$
We note that if $\pi$ is an isometric representation, then $\pi^*$ is also isometric, but in general $\pi^*$ need not be continuous. However,  for every reflexive (and hence super-reflexive) Banach space $\B$,  if $\pi$ is continuous, then so is $\pi^*$. 

Below,  we will need the following result from \cite{BFGM}:
\begin{proposition}\cite[Proposition 2.6]{BFGM}
\label{B-fixed + B' prop}
Let $\B$ be a super-reflexive Banach space, $G$ be a topological group and $\pi : G \rightarrow O(\B)$ be a continuous linear isometric representation.  Denote 
$\B ' (\pi)$ to be the annihilator of $(\B^{*})^{\pi^{*} (G)}$ in $\B$,  i.e.,  
$$\B ' (\pi) = \lbrace \xi \in \B : \forall \eta \in (\B^{*})^{\pi^{*} (G)},  \langle \xi,  \eta \rangle =0 \rbrace.$$
Then $\B = \B^{\pi (G)} \oplus \B ' (\pi)$.  
\end{proposition}

\subsection{Steinberg relations in $\SL_3 (\mathbb{Z})$}

For $1 \leq i,j \leq 3, i \neq j$ and $m \in \mathbb{Z}$, denote $e_{i,j} (m)$  to be the elementary matrix with $1$'s along the main diagonal, $m$ in the $(i,j)$-entry and $0$ in all other entries.  Using the convention $[a,b] = a^{-1} b^{-1} a b$,  the group $\SL_3 (\mathbb{Z})$ has the following relations that are called the \textit{Steinberg relations}:
\begin{enumerate}
\item For every $1 \leq i,j \leq 3, i \neq j$ and every $m_1,  m_2 \in \mathbb{Z}$, 
$$e_{i,j} (m_1) e_{i,j} (m_2) = e_{i,j} (m_1 + m_2).$$
\item For every $1 \leq i,j, k\leq 3,  \lbrace i,j,k \rbrace = \lbrace 1,2,3 \rbrace$ and every $m_1,  m_2 \in \mathbb{Z}$, 
$$[e_{i,j} (m_1), e_{j,k} (m_2) ] = e_{i,k} (m_1 m_2).$$
\item For every $1 \leq i,j, k\leq 3,  \lbrace i,j,k \rbrace = \lbrace 1,2,3 \rbrace$ and every $m_1,  m_2 \in \mathbb{Z}$, 
$$[e_{i,j} (m_1), e_{i,k} (m_2) ] =  [ e_{j,i} (m_1), e_{k,i} (m_2) ] = I.$$
\end{enumerate}

The group $\SL_3 (\mathbb{Z})$ has other relations that do not stem from the Steinberg relations.  Forgetting the other relations of $\SL_3 (\mathbb{Z})$ yields the Steinberg group $\St_{3} (\mathbb{Z})$.  Explicitly,  the Steinberg group $\St_{3} (\mathbb{Z})$ is the group generated by the set $S = \lbrace x_{i,j} : 1 \leq i, j \leq 3,  i \neq j \rbrace$ with the following relations: For every $m \in \mathbb{Z}$, denote $x_{i,j} (m) = x_{i,j}^m$.  With this notation,  the relations defining $\St_{3} (\mathbb{Z})$ are:
\begin{enumerate}
\item For every $1 \leq i,j \leq 3, i \neq j$ and every $m_1,  m_2 \in \mathbb{Z}$, 
$$x_{i,j} (m_1) x_{i,j} (m_2) = x_{i,j} (m_1 + m_2).$$
\item For every $1 \leq i,j, k\leq 3,  \lbrace i,j,k \rbrace = \lbrace 1,2,3 \rbrace$ and every $m_1,  m_2 \in \mathbb{Z}$, 
$$[x_{i,j} (m_1), x_{j,k} (m_2) ] = x_{i,k} (m_1 m_2).$$
\item For every $1 \leq i,j, k\leq 3,  \lbrace i,j,k \rbrace = \lbrace 1,2,3 \rbrace$ and every $m_1,  m_2 \in \mathbb{Z}$, 
$$[x_{i,j} (m_1), x_{i,k} (m_2) ] =  [ x_{j,i} (m_1), x_{k,i} (m_2) ] = I.$$
\end{enumerate}

\subsection{The Heisenberg group $\rm H_3 (\mathbb{Z})$}

The Heisenberg group $\rm H_3 (\mathbb{Z})$ is the group 
$$\rm H_3 (\mathbb{Z}) = \langle \left.  x,y, z \right\vert [x,y] = z,  [x,z] = e,  [y, z] =e \rangle.$$
Below, we will use the following relations for the Heisenberg group that are not hard verify: for every $k,m \in \mathbb{Z}$ it holds that 
$y^{-k} x^{m} y^{k} = x^m z^{km}$ and $x^{-k} y^{m} x^{k} = y^m z^{-km}$.

In the sequel,  we will use the fact that $\SL_3 (\mathbb{Z})$ (and $\St_3 (\mathbb{Z})$) contain several copies of  $\rm H_3 (\mathbb{Z})$. Explicitly,  for every $\lbrace i,j,k \rbrace = \lbrace 1,2,3 \rbrace$, if denote $\tilde{x} = e_{i,j} (1),  \tilde{y}= e_{j,k} (1) , \tilde{z} = e_{i,k} (1) \in \SL_3 (\mathbb{Z})$, 
then $\langle \tilde{x}, \tilde{y}, \tilde{z} \rangle < \SL_3 (\mathbb{Z})$ is isomorphic to $\rm H_3 (\mathbb{Z})$ (by the Steinberg relations) via the isomorphism $x \mapsto \tilde{x}, y\mapsto \tilde{y}, z \mapsto \tilde{z}$. 

\section{Banach property (T)}
\label{Banach prop T sec}


\subsection{Banach property $(T)$ for super-reflexive Banach spaces} 
\label{Banach property $(T)$ for super-reflexive Banach spaces}


Bader, Furman,  Gelander and Monod \cite{BFGM}  gave an equivalent version to Banach property (T) for super-reflexive spaces that is more convenient to work with than their general definition.  In \cite{BFGM},  it is shown that if $\B$ is a super-reflexive Banach space and $\pi : G  \rightarrow O( \B)$ is a linear isometric representation, then one can pass to a compatible norm on $\B$ in which $\B$ is uniformly convex and $\pi$ remains a linear isometric representation with respect to this new norm.  It follows that for a given topological group $G$ the following are equivalent:
\begin{itemize}
\item The group $G$ has property $(T_\B)$ for every uniformly convex Banach space $\B$.
\item The group $G$ has property $(T_\B)$ for every super-reflexive Banach space $\B$.
\end{itemize}  
Thus,  below we will focus on property $(T_\B)$ for uniformly convex Banach spaces $\B$ and the general result for super-reflexive Banach spaces will follow.  

For uniformly convex Banach spaces,  \cite{BFGM} gave the following equivalent definition for property $(T_\B)$:
\begin{definition}
\cite[Remark 2.11]{BFGM}
\label{BFGM def}
Let $\B$ be uniformly convex space and $G$ be a topological group.  Denote 
$\B ' (\pi)$ to be the annihilator of $(\B^{*})^{\pi^{*} (G)}$ in $\B$,  i.e.,  
$$\B ' (\pi) = \lbrace \xi \in \B : \forall \eta \in (\B^{*})^{\pi^{*} (G)},  \langle \xi,  \eta \rangle =0 \rbrace.$$

The group $G$ has property $(T_\B)$ if for every continuous linear isometric representation $\pi : G \rightarrow O (\B)$,  the restricted representation $\pi ' : G \rightarrow O (\B ' (\pi))$ does not have almost invariant vectors, i.e., there is a Kazhdan pair $(K, \varepsilon)$  (that depends on $\pi$) where $K \subseteq G$ is compact and $\varepsilon >0$ such that for every vector $\xi \in \B ' (\pi)$ it holds that
$$\sup_{g \in K} \Vert \pi  ' (g) \xi - \xi \Vert \geq \varepsilon \Vert \xi \Vert.$$
\end{definition}

\begin{observation}
\label{B^pi(G)=0 obs}
Let $G$ be a topological group.  By the above definition the following are equivalent:
\begin{enumerate}
\item The group $G$ has property $(T_\B)$ for every uniformly convex Banach space $\B$.
\item For every uniformly convex Banach space $\B$ and every isometric representation $\pi : G \rightarrow O( \B)$ with $\B^{\pi (G)} = \lbrace 0 \rbrace$, there is a compact set $K \subseteq G$ and a constant $\varepsilon >0$ such that for every unit vector $\xi \in \B$,  
$$\sup_{g \in K} \Vert \pi  (g) \xi - \xi \Vert \geq \varepsilon \Vert \xi \Vert.$$
\end{enumerate}
\end{observation}

\subsection{Relative Banach property (T)}
\label{Relative Banach property (T) subsec}


Here we introduce a variation of relative Banach property (T) called relative Banach property $(T^{\proj})$ that can be seen as a generalization of the definition of Banach property $(T^{\proj})$ given in \cite{LaatSalle2}.  We will show that relative Banach property $(T^{\proj})$ is a-priori stronger than the definition of relative Banach property (T) given in the introduction (we do not know if the two definitions do in fact coincide - see Remark \ref{rel T proj vs rel T remark} below).  In order to define relative Banach property $(T^{\proj})$,  we will need to first introduce some notation and terminology.

Let $G$ be a locally compact group with Haar measure $\mu$.  We denote $C_c (G)$ to be the compactly supported continuous functions $f: G  \rightarrow \mathbb{C}$ with the convolution product.  We further denote $\Prob_c (G) \subseteq C_c (G)$ to be functions $f: G \rightarrow [0, \infty)$ such that 
$\int_G f(g) d \mu (g) =1$.  Given a continuous representation $\pi : G \rightarrow \GL (\B)$ where $\B$ is a Banach space,  we define for every $f \in C_c (G)$ an operator $\pi (f)$ via the Bochner integral
$$\pi (f)\xi = \int_G f(g) \pi  (g) \xi d \mu (g),  \forall \xi \in \B.$$

For a class of Banach spaces $\mathcal{E}$,  we denote $\mathcal{U} (G, \mathcal{E})$ to be the class of all continuous isometric linear representations $(\pi, \B)$ where $\B \in \mathcal{E}$.   When $G$ is obvious from the context,  we will denote $\mathcal{U} ( \mathcal{E}) = \mathcal{U} (G, \mathcal{E})$.  We define a norm $\Vert . \Vert_{\mathcal{U} (G, \mathcal{E})}$ on $C_c (G)$ by 
$$\Vert f \Vert_{\mathcal{U} (G, \mathcal{E})} = \sup_{\pi \in \mathcal{U} (G, \mathcal{E})} \Vert \pi (f) \Vert,$$
and denote $C_{\mathcal{U} (G, \mathcal{E})} (G)$ to be the completion of $C_c (G)$ with respect to this norm.  We note that for every $f \in C_{\mathcal{U} (G, \mathcal{E})} (G)$ and every $\pi \in \mathcal{U} (G, \mathcal{E})$,  the operator $\pi (f) \in B (\B)$ is well-defined as a limit of operators $\pi (f_n)$ with $f_n \in C_c (G)$.

\begin{definition}
Let $G$ be a locally compact group with a subgroup $H < G$.  We will say that $(G,H)$ has relative property $(T_{\mathcal{E}}^\proj)$ if there is a sequence $h_n \in \Prob_c (G)$ that converges to $f \in C_{\mathcal{U} (G, \mathcal{E})} (G)$ (with respect to the norm $\Vert . \Vert_{\mathcal{U} (G, \mathcal{E})}$) such that for every $(\pi, \B) \in \mathcal{U} (G, \mathcal{E})$,  $\im (\pi (f)) \subseteq \B^{\pi (H)}$.  
\end{definition}

We show that relative property $(T^{\proj}_{\mathcal{E}})$ imply relative property $(T_{\B})$ for every $\B \in \mathcal{E}$ as defined above (see Definition \ref{relative property T Jol def}):
\begin{proposition}
\label{relative T^proj implies relative T prop}
Let $G$ be a locally compact group,  $H < G$ a subgroup and $\mathcal{E}$ a class of Banach spaces.  Assume that $(G,H)$ has relative property $(T^{\proj}_{\mathcal{E}})$.  Then for every $\B \in \mathcal{E}$,  the pair $(G,H)$ has relative property $(T_{\B})$.
\end{proposition}

\begin{proof}
We need to show that for every $\gamma >0$,  there are $K \subseteq G$ compact and $\varepsilon >0$, such that for every $(\pi, \B) \in \mathcal{U} (\mathcal{E})$ and every unit vector $\xi \in \B$,  if $\sup_{g \in K} \Vert \pi  (g) \xi - \xi \Vert < \varepsilon$,  then there is $\eta \in \B^{\pi (H)}$ such that $\Vert \xi - \eta \Vert < \gamma$.  

Let $\gamma >0$ arbitrary.  By the assumption that $(G,H)$ has relative property $(T^{\proj}_{\mathcal{E}})$ it follows that there is a sequence $h_n \in \Prob_c (G)$ that converges to $f \in C_{\mathcal{U} (\mathcal{E})} (G)$ such that for every $(\pi, \B) \in  \mathcal{U} (\mathcal{E})$,  $\im (\pi (f)) \subseteq \B^{\pi (H)}$.  

For $f \in C_{\mathcal{U} (\mathcal{E})} (G)$ as above, there is $h \in \Prob_c (G)$ such that $\Vert f-h \Vert_{\mathcal{U} (\mathcal{E})} < \frac{\gamma}{2}$.  We take $K$ to be a compact set such that $\supp (h) \subseteq K$ and $\varepsilon = \frac{\gamma}{2}$ and show that this choice of $K, \varepsilon$ fulfils the needed condition.

Indeed,  for every $(\pi, \B) \in \mathcal{U} (\mathcal{E})$ and every unit vector $\xi \in \B$,  if $\sup_{g \in K} \Vert \pi  (g) \xi - \xi \Vert < \varepsilon = \frac{\gamma}{2}$,  then for $\eta = \pi (f) \xi \in \B^{\pi (H)}$ it holds that
\begin{dmath*}
\Vert \xi - \eta \Vert \leq \Vert \xi - \pi  (h) \xi \Vert + \Vert \pi  (h) \xi - \eta \Vert <
\left\Vert \int_{G} h (g) (\xi - \pi  (g) \xi) d \mu (g) \right\Vert +  \frac{\gamma}{2} \leq \\
{\int_{G} h (g) \left\Vert  \xi - \pi  (g) \xi \right\Vert d \mu (g)}  +  \frac{\gamma}{2} \leq 
\max_{g \in \supp (h)}  \left\Vert  \xi - \pi  (g) \xi \right\Vert + \frac{\gamma}{2} \leq 
\gamma
\end{dmath*}
as needed.
\end{proof}

\begin{remark}
\label{rel T proj vs rel T remark}
We do not know if the opposite direction of the above proposition is also true,  i.e.,  if relative property $(T_{\B})$ for every $\B \in \mathcal{E}$ implies property $(T^{\proj}_{\mathcal{E}})$.  The problem is that even in the classical setting of Hilbert spaces there is not natural candidate for the sequence $h_n \in \Prob_c (G)$.  To illustrate this,  we consider what should be a simple situation: Let $G$ be a finitely generating group with a finite generating set $S$ and $N \triangleleft G$ a normal subgroup and $\mathcal{H}$ be the class of all Hilbert spaces.  

We recall that from that fact that $N$ is a normal subgroup it follows for every unitary $(\pi,  \mathbb{H})$ of $G$ on a Hilbert space $\mathbb{H}$ the subspaces $\mathbb{H}^{\pi (N)},  (\mathbb{H}^{\pi (N)})^{\perp}$ are $G$-invariant subspaces with respect to the $G$ action.  In this setting,  relative property $(T)$ for $(G,N)$ can be described by the following formulation in \cite[Theorem 1.2 (b2)]{Jol}: There is $\varepsilon_0 >0$ such that for every $\alpha >0$,  every unitary representation $(\pi,  \mathbb{H})$ of $G$ on a Hilbert space $\mathbb{H}$ and every unit vector $\xi \in \mathbb{H}$, if 
$$\max_{s \in S} \Vert \pi (s) \xi - \xi \Vert \leq \alpha \varepsilon_0$$
then $\Vert \xi - P_{\mathbb{H}^{\pi (N)} }\xi \Vert \leq \alpha$ where $P_{\mathbb{H}^{\pi (N)} }$ is the orthogonal projection on $\mathbb{H}^{\pi (N)}$.  

Our naive guess for $h_n \in \Prob_c (G)$ is the sequence 
$$h_n = \left(\frac{1}{2} I + \frac{1}{2 \vert S \vert} \sum_{s \in S} s \right)^n ,$$
(which is the sequence that converges to a Kazhdan projection when $N = G$).   For every unitary representation $(\pi,  \mathbb{H})$,  $\left. \pi \right\vert_{(\mathbb{H}^{\pi (N)})^{\perp}} (h_n)$ indeed converges to $0$ and the rate of convergence can be bounded independently of $\pi$.  However,  we see no reason that $\left. \pi \right\vert_{\mathbb{H}^{\pi (N)}} (h_n)$ will converge when $N \neq G$ and thus (as far as we can tell) this naive attempt fails.

\end{remark}


\subsection{Hereditary properties of property $(T_\B)$}
\label{Hereditary properties of property (T_B) subsec}

Lafforgue showed that property $(T_{\B})$ is inherited by lattices via an induction of representation:
\begin{proposition}
\label{prop T passed to lattice prop}
\cite[Proposition 4.5]{Laff1}, \cite[Proposition 5.3]{Laff2}
Let $G$ be a locally compact group,  $\Gamma < G$ a lattice and $\mathcal{E}$ a class of Banach spaces.  Also let $\mathcal{E} '$ be a class of Banach spaces such that for every $\B \in \mathcal{E}$ and every finite measure space $(X, \mu)$ it holds that $L^2 (X, \mu ; \B) \in \mathcal{E} '$.  If $G$ has property $(T_{\B '})$ for every $\B ' \in  \mathcal{E} '$ , then $\Gamma$ has property $(T_{\B})$ for every $\B \in \mathcal{E}$. 
\end{proposition}

\begin{remark}
The above formulation differs from the formulation in \cite{Laff1,  Laff2} since we do not assume that $\mathcal{E}$ is closed under passing to $L^2$-sums.
\end{remark}

\begin{corollary}
\label{passing to lattices coro}
Let $G$ be a locally compact group and $\Gamma < G$ a lattice.  If $G$ has property $(T_\B)$ for every uniformly convex Banach space $\B$,  then $\Gamma$ has  property $(T_\B)$ for every super-reflexive Banach space $\B$.
\end{corollary}

\begin{proof}
This follows immediately from Proposition \ref{prop T passed to lattice prop} and Theorem \ref{L^2 sum of uc spaces thm}.
\end{proof}

\section{Bounded generation and Banach property (T)}
\label{Bounded generation and Banach property (T) sec}

In this section,  we adapt a bounded generation argument of Shalom \cite{Shalom1} to our setting and show that,  in our setting,  relative Banach property (T) and bounded generation imply Banach property (T).  

\begin{definition}
Let $G$ be a group with subgroups $H_1,...,H_k$.  We say that $H_1,...,H_k$ boundedly generate $G$ if there is a number $\nu = \nu (H_1,...,H_k) \in \mathbb{N}$ such that every element $g \in G$ can be written by at most $\nu$ elements of $H_1 \cup ... \cup H_k$. 
\end{definition}

\begin{lemma}
\label{bounded orbit lemma}
Let $G$ be a group with subgroups $H_1,...,H_k$ that boundedly generate $G$ and denote $\nu = \nu (H_1,...,H_k)$ as above.  Also,  let $\pi :G \rightarrow O (\B)$ be a continuous linear isometric representation.  Assume that there are $\eta_1,..., \eta_k \in \B$ such that for every $1 \leq i \leq k$,  $\eta_i \in \B^{\pi (H_i)}$.   Then for every $\xi \in \B$ and every $g \in G$, 
$$\Vert \pi (g) \xi - \xi \Vert \leq 2 \nu \max_{1 \leq i \leq k} \Vert \xi - \eta_i \Vert.$$
\end{lemma}

\begin{proof}
Let $g \in G$ such that $g = g_1 ...g_j$ with $g_1,...,g_j \in \bigcup_{i=1}^k H_i$.  We will prove by induction that for every $\xi \in \B$,
\begin{equation}
\label{ind.  ineq}
\Vert \pi (g) \xi - \xi \Vert \leq 2 j \max_{1 \leq i \leq k} \Vert \xi - \eta_i \Vert.
\end{equation}

For $j=1$,  there is $1 \leq i_0 \leq k$ such that $g \in H_{i_0}$.  Then 
\begin{align*}
\Vert \pi (g) \xi - \xi \Vert = \Vert \pi (g) \xi - \pi (g) \eta_{i_0} + \eta_{i_0} - \xi \Vert \leq 
 \Vert \pi (g) (\xi - \eta_{i_0}) \Vert + \Vert \eta_{i_0} - \xi \Vert = \\
 2 \Vert \xi - \eta_{i_0} \Vert \leq 2 \max_{1 \leq i \leq k} \Vert \xi - \eta_i \Vert.
\end{align*}

Assume \eqref{ind.  ineq} holds for $j$ and let $g = g_{1} ... g_{j+1}$ with $g_1,...,g_{j+1} \in \bigcup_{i=1}^k H_i$. Then for every $\xi \in \B$, 
\begin{align*}
\Vert \pi (g) \xi - \xi \Vert = \Vert \pi (g_1 ... g_{j} g_{j+1}) \xi - \pi (g_1 ... g_{j}) \xi + \pi (g_1 ... g_{j}) \xi - \xi \Vert \leq \\
\Vert \pi (g_1 ... g_{j}) (\pi (g_{j+1}) \xi - \xi)\Vert + \Vert \pi (g_1 ... g_{j}) \xi - \xi \Vert = \\
\Vert  (\pi (g_{j+1}) \xi - \xi)\Vert + \Vert \pi (g_1 ... g_{j}) \xi - \xi \Vert \leq^{\text{The induction assumption}} \\
2  \max_{1 \leq i \leq k} \Vert \xi - \eta_i \Vert + 2 j \max_{1 \leq i \leq k} \Vert \xi - \eta_i \Vert = 2 (j+1) \max_{1 \leq i \leq k} \Vert \xi - \eta_i \Vert.
\end{align*}

By the assumption of bounded generation,  every $g \in G$ can be written as $g = g_1 ... g_\nu$ with $g_1,...,g_\nu \in \bigcup_{i=1}^k H_i$ and thus it follows that for every $g \in G$ and every $\xi \in \B$, 
$$\Vert \pi (g) \xi - \xi \Vert \leq 2 \nu \max_{1 \leq i \leq k} \Vert \xi - \eta_i \Vert,$$
as needed.
\end{proof}

\begin{theorem}
\label{bounded generation imply prop T thm}
Let $G$ be a locally compact group and $H_1,...,H_k <G$ subgroups that boundedly generate $G$.  If the pairs $(G,H_1),...,(G,H_k)$ has relative property $(T_\B)$ for every uniformly convex Banach space $\B$,  then $G$ has property $(T_\B)$ for every uniformly convex Banach space $\B$.
\end{theorem}

\begin{proof}
By Observation \ref{B^pi(G)=0 obs},  we need to show that for every uniformly convex Banach space $\B$ and every $\pi: G \rightarrow O ( \B)$ with $\B^{\pi (G)} = \lbrace 0 \rbrace$ there is a compact set $K \subseteq G$ and $\varepsilon >0$ such that for every unit vector $\xi \in \B$, 
$$\sup_{g \in K} \Vert \pi  (g) \xi - \xi \Vert \geq \varepsilon.$$

Denote $\nu = \nu (H_1,...,H_k) \in \mathbb{N}$ as in the definition above.   By assumption, there are compact sets $K_1,...,K_k$ and constants $\varepsilon_1,..., \varepsilon_k>0$ such that for every $i =1,...,k$ and every unit vector $\xi \in \B$ if 
$$\sup_{g \in K_i} \Vert  \pi  (g) \xi - \xi \Vert < \varepsilon_i,$$ 
then there is $\eta_i \in \B^{\pi (H_i)}$ such that 
$$\Vert \xi - \eta_i \Vert \leq \frac{1}{4 (\nu+1)}.$$

Denote $K = \bigcup_{i=1}^k K_i$ and $\varepsilon = \min \lbrace \varepsilon_1,...,\varepsilon_k \rbrace$.  We will show that for this choice of $K, \varepsilon$ it holds for every unit vector $\xi \in \B$ that 
$$\sup_{g \in K} \Vert  \pi  (g) \xi - \xi \Vert \geq \varepsilon.$$

Assume towards contradiction that there is a unit vector $\xi \in \B $ such that
$$\sup_{g \in K} \Vert \pi  (g) \xi - \xi \Vert < \varepsilon.$$

Thus,  for every $i =1,...,k$ there is $\eta_i \in \B^{\pi (H_i)}$ such that $\Vert \xi - \eta_i \Vert < \frac{1}{4 (\nu+1)}$.  Applying Lemma \ref{bounded orbit lemma},  it follows that for every $g \in G$, 
$$\Vert \pi  (g) \xi - \xi \Vert \leq \frac{\nu}{2(\nu+1)} \leq \frac{1}{2}.$$

Thus the orbit of $\xi$ in $\B $ under the action of $G$ is contained in a closed ball of radius $\frac{1}{2}$ around $\xi$.  Denote $C$ to be the closure of the convex hull of the orbit of $\xi$.  Recall that $\xi$ is a unit vector and thus $0 \notin C$.  By uniform convexity there is a unique vector with a minimal norm in $C$ and thus this vector is fixed by the $\pi$ action of $G$.  It follows that $C \cap \B^{\pi (G)} \neq \emptyset$ which contradicts the assumption that $\B^{\pi (G)} \neq \lbrace 0 \rbrace$. 
\end{proof}

\section{Averaging operations on $\rm H_3 (\mathbb{Z})$}
\label{Averaging operations sec}

In this section,  we will prove norm bounds on averaging operations on the Heisenberg group that are needed in our proof of relative Banach property (T) stated in the introduction. 

For every $k \in \mathbb{N} \cup \lbrace 0 \rbrace$,  we define $X_k,  Y_k,  Z_k  \in \Prob_c (\rm H_3 (\mathbb{Z}))$ by 
$$X_k= \frac{e+x^{2^k}}{2},  Y_k = \frac{e+y^{2^k}}{2}, Z_k = \frac{e+z^{2^k}}{2}.$$

Also,  for $d \in \mathbb{N}$,  we define 
$$X^d = \frac{1}{2^{d}} \sum_{a=0}^{2^{d}-1} x^a, Y^d = \frac{1}{2^{d}} \sum_{b=0}^{2^{d}-1} y^b, Z^d = \frac{1}{2^{d}} \sum_{c=0}^{2^{d}-1} z^c.$$

\begin{observation}
\label{product identity obs}
For every $d \in \mathbb{N}$ it holds that
$$X^d = \prod_{a=0}^{d-1} X_a, Y^d = \prod_{b=0}^{d-1} Y_b, Z^d = \prod_{c=0}^{d-1} Z_c.$$
\end{observation}

\begin{theorem}
\label{change of order thm}
Let $d \in \mathbb{N}$ be a constant and $A,B \subseteq \lbrace 0,...,d-1 \rbrace$ be sets such that $(\max A) (\max B) \leq d-2$,  then for every Banach space $\B$ and every isometric linear representation $\pi :  H_3 (\mathbb{Z}) \rightarrow O (\B)$ it holds that 
$$\left\Vert \pi \left( \left( \left( \prod_{a \in A} X_a \right) \left( \prod_{b \in B} Y_b \right) - \left( \prod_{b \in B} Y_b \right) \left( \prod_{a \in A} X_a \right)   \right) Z^d \right)   \right\Vert \leq 8 \left( \frac{1}{2} \right)^{d-\max A - \max B}.$$

In particular,  for $d_1, d_2, d_3 \in \mathbb{N} \cup \lbrace 0 \rbrace$,  if $d_1+d_2 \leq d_3-2$,  then for any class of Banach spaces $\mathcal{E}$ it holds that 
$$\left\Vert \left( X^{d_1} Y^{d_2} - Y^{d_2} X^{d_1}    \right) Z^d    \right\Vert_{\mathcal{U} (\mathcal{E})} \leq 8 \left( \frac{1}{2} \right)^{d_3- (d_1+d_2)}.$$
\end{theorem}

\begin{proof}
We note that 
\begin{align*}
\left( \prod_{a \in A} X_a \right) \left( \prod_{b \in B} Y_b \right) - \left( \prod_{b \in B} Y_b \right) \left( \prod_{a \in A} X_a \right)  = \\
\frac{1}{2^{\vert A \vert + \vert B \vert}} \sum_{f \in \lbrace 0,1 \rbrace^A,  h \in \lbrace 0,1 \rbrace^B}  x^{\sum_{k \in A} f(k) 2^k} y^{\sum_{l \in B} h(l) 2^l}  - y^{\sum_{l \in B} h(l) 2^l} x^{\sum_{k \in A} f(k) 2^k}  = \\
\frac{1}{2^{\vert A \vert + \vert B \vert}} \sum_{f \in \lbrace 0,1 \rbrace^A,  h \in \lbrace 0,1 \rbrace^B}  x^{\sum_{k \in A} f(k) 2^k} y^{\sum_{l \in B} h(l) 2^l}  (e-z^{(\sum_{k \in A} f(k) 2^k ) (\sum_{l \in B} h(l) 2^l)}). 
\end{align*}

Note that for every $f \in \lbrace 0,1 \rbrace^A,  h \in \lbrace 0,1 \rbrace^B$ it holds that 
$$(\sum_{k \in A} f(k) 2^k ) (\sum_{l \in B} h(l) 2^l) \leq 2^{\max A + \max B +2}.$$
Thus, it is enough to show that for every $1 \leq m \leq 2^{\max A + \max B +2}$ it holds that 
$$\Vert \pi ((e-z^m) Z^d) \Vert \leq \frac{2m}{2^d} \leq \left( \frac{1}{2} \right)^{d-\max A - \max B -3},$$
but this follows immediately from the fact that 
$$(e-z^m) Z^d = \frac{1}{2^d} \left(\sum_{k=0}^{m-1} z^k - \sum_{k=2^d}^{2^d+m-1} z^k \right).$$ 
\end{proof}

\begin{lemma}
\label{product inequality lemma}
Let $\delta_0 : (0,2] \rightarrow (0,1]$ be a function.  There is a constant $0 \leq r_0  <1$ such that for every $(\pi, \B) \in \mathcal{U} (\mathcal{E}_{uc} (\delta_0) )$,  every  $k, m \in \mathbb{N} \cup \lbrace 0 \rbrace$ and every  $\zeta \in \B$,  if $\Vert \pi (e-z^{m} ) \zeta \Vert \geq \frac{1}{2} \Vert \zeta \Vert$,  then 
$$\left\Vert \pi \left( \frac{e+x z^k}{2} \frac{e+y^m}{2} \right) \zeta \right\Vert \leq r_0 \Vert \zeta \Vert.$$
\end{lemma}

\begin{proof}
We will show that the needed inequality holds for $r_0 = r_0 (\delta_0, \frac{1}{2})$ where this is the constant of Corollary \ref{uc ineq coro}.

Fix $(\pi, \B) \in \mathcal{U} (\mathcal{E}_{uc} (\delta_0) )$,  $k,m \in \mathbb{N} \cup \lbrace 0 \rbrace$ and $\zeta \in \B$ such that $\Vert \pi (e-z^{m} ) \zeta \Vert \geq \frac{1}{2} \Vert \zeta \Vert$.

We note that
\begin{align*}
\frac{e+x z^k}{2} \frac{e+y^{m}}{2}  = 
\frac{1}{2} \left( \frac{e+x z^k}{2} \right) + \frac{1}{2} \left( \frac{y^{m}+x y^{m} z^k}{2}  \right) = \\
\frac{1}{2} \left( \frac{e+x z^k}{2} \right) + \frac{y^m}{2} \left( \frac{e+y^{-m} x y^{m} z^k}{2}  \right) = 
 \frac{1}{2} \left( \frac{e+x z^k}{2} \right) + \frac{y^{m}}{2} \left( \frac{e+x z^k z^{m}}{2} \right).
\end{align*}
Thus
\begin{align*}
\left\Vert \pi \left( \frac{e+x z^k}{2} \frac{e+y^m}{2} \right) \zeta \right\Vert \leq 
\frac{1}{2} \left\Vert \pi \left(\frac{e+x z^k}{2}  \right) \zeta \right\Vert + \frac{1}{2} \left\Vert \pi (y^m) \pi \left(\frac{e+x z^k z^m}{2}  \right) \zeta \right\Vert \leq^{\Vert \pi (y^m) \Vert =1} \\
\frac{1}{2} \left\Vert \frac{I+\pi (x z^k)}{2}  \zeta \right\Vert + \frac{1}{2} \left\Vert  \frac{I+ \pi (x z^k) \pi (z^m)}{2}  \zeta \right\Vert.
\end{align*}
Denote $T = \pi (x z^k)$ and $S = \pi (z^m)$.  Note that $T,S \in O (\B)$ are commuting operators and that $\Vert (I-S) \zeta \Vert \geq \frac{1}{2} \Vert \zeta \Vert$.  Thus the conditions of Corollary \ref{uc ineq coro} and it follows that 
$$\frac{1}{2} \left\Vert \frac{I+\pi (x z^k)}{2}  \zeta \right\Vert + \frac{1}{2} \left\Vert  \frac{I+ \pi (x z^k) \pi (z^m)}{2}  \zeta \right\Vert = \frac{1}{2} \left\Vert \frac{I+T}{2} \zeta \right\Vert + \frac{1}{2} \left\Vert \frac{I+TS}{2} \zeta \right\Vert \leq r_0 \Vert \zeta \Vert,$$
as needed.
\end{proof}

\begin{lemma}
\label{X,Y prod ineq 2 lemma}
Let $\delta_0 : (0,2] \rightarrow (0,1]$ be a function.  There is a constant $0 \leq r_1  <1$ such that for every $(\pi, \B) \in \mathcal{U} (\mathcal{E}_{uc} (\delta_0) )$,  every  $n,  m \in \mathbb{N}$ such that $1 \leq m \leq 2^{n-1}$ and every  $\zeta \in \B$,  if $\Vert \pi (e-z^{m} ) \zeta \Vert \geq \frac{1}{2} \Vert \zeta \Vert$,  then 
$$\left\Vert \pi \left( X_0 Y^n \right) \zeta  \right\Vert \leq r_1 \Vert \zeta \Vert.$$
\end{lemma}

\begin{proof}
Fix $(\pi, \B) \in \mathcal{U} (\mathcal{E}_{uc} (\delta_0) )$ and $n,m \in \mathbb{N}$ as above.  Let $0 \leq r_0  < 1$ be the constant of Lemma \ref{product inequality lemma}.  We will show that for $r_1 = \frac{r_0 +1}{2}$ the needed inequality holds. 

Fix $\zeta \in \B$ such that $\Vert \pi (e-z^{m} ) \zeta \Vert \geq \frac{1}{2} \Vert \zeta \Vert$.
 
Denote $t =  \lfloor \frac{2^{n}}{2m} \rfloor$ and $k = 2^{n} - 2mt$ and note that $t \geq 1,0 \leq k \leq \min \lbrace 2^{n-1}, 2m-1 \rbrace$ and  $2^{n} = 2mt + k$.  It follows that 
\begin{align*}
Y^n = \frac{1}{2^n} \sum_{b=0}^{2^n-1} y^b =   
\frac{2mt}{2^n} \left(\frac{1}{2mt} \sum_{b=0}^{2mt-1} y^b \right) + \frac{k}{2^n} \left(\frac{1}{k} \sum_{b=2mt}^{2^n-1} y^b \right).
\end{align*}

We claim it is sufficient to prove that 
\begin{equation}
\label{thm prod ineq}
\left\Vert \pi \left( X_0 \left(\frac{1}{2mt} \sum_{b=0}^{2mt-1} y^b \right)  \right) \zeta \right\Vert \leq r_0 \Vert \zeta \Vert.
\end{equation}

Indeed,  note that 
$$\left\Vert \pi \left( X_0 \left(\frac{1}{k} \sum_{b=2mt}^{2^n-1} y^b \right) \right) \right\Vert \leq 1$$
and thus if \eqref{thm prod ineq} holds,  then 
\begin{align*}
\left\Vert \pi \left( X_0 Y^n \right) \zeta  \right\Vert \leq  \frac{2mt}{2^n} \left\Vert \pi \left( X_0 \left( \frac{1}{2mt} \sum_{b=0}^{2mt-1} y^b \right) \right) \zeta \right\Vert + \frac{k}{2^n} \left\Vert \pi \left( X_0 \left( \frac{1}{k} \sum_{b=2mt}^{2^n-1} y^b \right) \right) \zeta  \right\Vert \leq \\
 \left( \frac{2mt}{2^n} r_0 + \frac{k}{2^n} \right) \Vert \zeta \Vert \leq^{k \leq 2^{n-1}} \frac{r_0 + 1}{2} \Vert \zeta \Vert,
\end{align*}
as needed. 

We will finish the proof by proving \eqref{thm prod ineq}.  We note that
\begin{align*}
\frac{1}{2mt} \sum_{b=0}^{2mt-1} y^b  = \left( \frac{1}{m} \sum_{b_1 =0}^{m-1} y^{b_1} \right) \left( \frac{1}{t} \sum_{b_2=0}^{t-1} y^{2m b_2} \right) \left(\frac{e + y^m}{2} \right).
\end{align*}

Thus,
\begin{align*}
X_0 \left( \frac{1}{2mt} \sum_{b=0}^{2mt-1} y^b \right) = \frac{e+x}{2}  \left( \frac{1}{m} \sum_{b_1 =0}^{m-1} y^{b_1} \right) \left( \frac{1}{t} \sum_{b_2=0}^{t-1} y^{2m b_2} \right) \left(\frac{e + y^m}{2} \right) = \\
\frac{1}{m t} \sum_{b_1 =0}^{m-1}  \sum_{b_2=0}^{t-1} y^{b_1 +2m b_2} \frac{e+y^{-(b_1 +2m b_2)} x y^{b_1 +2m b_2}}{2}  \left(\frac{e + y^m}{2} \right) = \\\frac{1}{m t} \sum_{b_1 =0}^{m-1}  \sum_{b_2=0}^{t-1} y^{b_1 +2m b_2} \frac{e+ x z^{b_1 +2m b_2}}{2}  \left(\frac{e + y^m}{2} \right).
\end{align*}

Using the fact that $\Vert \pi (y^{b_1 +2m b_2}) \Vert =1$ for every $b_1, b_2 \in \mathbb{N}$, it follows that
\begin{align*}
& \left\Vert \pi \left( X_0 \left( \frac{1}{2mt} \sum_{b=0}^{2mt-1} y^b \right)  \right) \zeta \right\Vert \leq \\
& \frac{1}{m t} \sum_{b_1 =0}^{m-1}  \sum_{b_2=0}^{t-1} \left\Vert \pi \left( \frac{e+ x z^{b_1 +2m b_2}}{2} \right) \pi  \left(\frac{e + y^m}{2} \right) \zeta \right\Vert \leq^{\text{Lemma } \ref{product inequality lemma}}  
  \frac{1}{m t} \sum_{b_1 =0}^{m-1}  \sum_{b_2=0}^{t-1} r_0 \Vert \zeta \Vert = r_0 \Vert \zeta \Vert,
\end{align*}
and the proof of \eqref{thm prod ineq} is concluded.
\end{proof}

\begin{theorem}
Let $\delta_0 : (0,2] \rightarrow (0,1]$ be a function.  Let $0 \leq r_1  <1$ be the constant given in Lemma \ref{X,Y prod ineq 2 lemma} above.   For every $(\pi, \B) \in \mathcal{U} (\mathcal{E}_{uc} (\delta_0) )$,  every  $n \in \mathbb{N}, n \geq 2$  and every  $\xi \in \B$,  it holds that
$$\left\Vert \pi \left( X_0 Y^n (e-Z_0) \right)  \xi  \right\Vert \leq  \max \left\lbrace r_1  \Vert \pi (e-Z_0) \xi \Vert, \frac{1}{2^{n-2}} \Vert \xi \Vert \right\rbrace.$$
\end{theorem}

\begin{proof}
Fix $(\pi, \B) \in \mathcal{U} (\mathcal{E}_{uc} (\delta_0) )$,  $n \in \mathbb{N}, n \geq 2$ and $\xi \in \B$.  If 
$ \Vert \pi (e-Z_0) \xi \Vert \leq \frac{1}{2^{n-2}} \Vert \xi \Vert$,  then 
$$\left\Vert \pi \left( X_0 Y^n  (e-Z_0) \right) \xi  \right\Vert \leq \Vert \pi (e-Z_0) \xi \Vert \leq \frac{1}{2^{n-2}} \Vert \xi \Vert,$$
and we are done. 

Assume that $ \Vert \pi (e-Z_0) \xi \Vert > \frac{1}{2^{n-2}} \Vert \xi \Vert$.  Then 
\begin{align*}
\frac{1}{2^{n-1}} \sum_{m =0}^{2^{n-1}-1} \Vert \pi ((e - z^{m})(e-Z_0)) \xi \Vert \geq 
\left\Vert \pi \left ( \left( e - \frac{1}{2^{n-1}} \sum_{m =0}^{2^{n-1}-1} z^m \right) (e-Z_0) \right) \xi \right\Vert \geq \\
 \Vert \pi (e-Z_0) \xi \Vert - \left\Vert \pi \left( \left( \frac{1}{2^{n-1}} \sum_{m =0}^{2^{n-1}-1} z^m  \right) (e-Z_0) \right) \xi \right\Vert \geq^{e-Z_0 = \frac{e-z}{2}} \\
 \left\Vert \pi (e-Z_0) \xi \right\Vert - \frac{1}{2^{n-1}} \left\Vert  \pi \left( \frac{e- z^{2^{n-1}}}{2} \right)  \xi \right\Vert \geq \\
 \left\Vert \pi (e-Z_0) \xi \right\Vert - \frac{1}{2^{n-1}}  \Vert \xi \Vert >^{ \Vert \pi (e-Z_0) \xi \Vert > \frac{1}{2^{n-2}} \Vert \xi \Vert} \\
  \left\Vert \pi (e-Z_0) \xi \right\Vert  - \frac{1}{2} \left\Vert \pi (e-Z_0) \xi \right\Vert  = \frac{1}{2} \left\Vert \pi (e-Z_0) \xi \right\Vert.
\end{align*}

It follows that for $\zeta = \pi (e-Z_0) \xi$,  there is $1 \leq m \leq 2^{n-1}-1$ such that $\Vert \pi (e-z^m) \zeta \Vert \geq \frac{1}{2} \Vert \zeta \Vert$.  Thus, by Lemma \ref{X,Y prod ineq 2 lemma},  
$$\Vert \pi \left( X_0 Y^n \right) \zeta \Vert \leq r_1 \Vert \zeta \Vert,$$
i.e., 
$$\left\Vert \pi \left( X_0 Y^n (e-Z_0) \right)  \xi  \right\Vert \leq  r_1  \Vert \pi (e-Z_0) \xi \Vert ,$$
as needed.
\end{proof}

\begin{corollary}
\label{X,Y prod ineq final coro}
Let $\delta_0 : (0,2] \rightarrow (0,1]$ be a function. Let $0 \leq r_1  <1$ be the constant given in Lemma \ref{X,Y prod ineq 2 lemma} above.   For every $(\pi, \B) \in \mathcal{U} (\mathcal{E}_{uc} (\delta_0) )$,  every  $n \in \mathbb{N}, n \geq 2$,   every $a_0,b_0 \in \mathbb{N} \cup \lbrace 0 \rbrace$  and every  $\xi \in \B$,  it holds that
$$\left\Vert \pi \left( \left( X_{a_0} \left( \prod_{b= 0}^{n-1} Y_{b_0+b} \right)\right) (e- Z_{a_0 + b_0}) \right) \xi \right\Vert \leq \max \left\lbrace r_1  \Vert \pi (e-Z_{a_0 + b_0}) \xi \Vert, \frac{1}{2^{n-2}} \Vert \xi \Vert \right\rbrace.$$
\end{corollary}

\begin{proof}
Fix $(\pi, \B),  n, a_0,b_0$ as above. 

Let $H < H_3 (\mathbb{Z})$ be the subgroup $H = \langle x^{2^{a_0}},  y^{2^{b_0}} \rangle$.  
We note that $H$ is isomorphic to $H_3 (\mathbb{Z})$ via the isomorphism $\Phi :  H_3 (\mathbb{Z}) \rightarrow H$ induced by $\Phi (x) = x^{2^{a_0}},  \Phi (y) = y^{2^{b_0}}$.  Note that (by extending $\Phi$ linearly)
$$\Phi (X_0) = X_{a_0},  \Phi (Y^n) = \prod_{b= 0}^{n-1} Y_{b_0+b}.$$
Also note that
$$\Phi (z) = \Phi (x^{-1} y^{-1} x y) = x^{-2^{a_0}} y^{- 2^{b_0}} x^{2^{a_0}}  y^{ 2^{b_0}} = z^{ 2^{a_0 + b_0}},$$
and thus $\Phi (Z_0) = Z_{a_0 + b_0}$.   

Define a new representation $(\pi_0, \B)$ of $H_3 (\mathbb{Z})$ by $\pi_0 = \pi \circ \Phi$.  Let $\xi \in \B$,  then
\begin{align*}
\left\Vert \pi \left( \left( X_{a_0} \left( \prod_{b=0}^{n-1} Y_{b_0+b} \right) \right) (e- Z_{a_0 + b_0}) \right) \xi \right\Vert = 
 \left\Vert \pi_0 \left( \left( X_{0} Y^n \right) (e- Z_{0}) \right) \xi \right\Vert \leq \\
  \max \left\lbrace r_1  \Vert \pi_0 (e-Z_{0}) \xi \Vert, \frac{1}{2^{n-2}} \Vert \xi \Vert \right\rbrace = 
 \max \left\lbrace r_1  \Vert \pi (e-Z_{a_0 + b_0}) \xi \Vert, \frac{1}{2^{n-2}} \Vert \xi \Vert \right\rbrace ,  
\end{align*}
as needed.
\end{proof}

\begin{theorem}
\label{main ineq thm}
Let $\delta_0 : (0,2] \rightarrow (0,1]$ be a function.  There are constants $0 \leq r_2 <1, C>0$ such that for  every $d_1, d_2, d_3 \in \mathbb{N}$ with $d_1, d_2 \leq d_3,  d_1 + d_2 \geq d_3$,  it holds that 
$$\left\Vert X^{d_1} Y^{d_2} \left( Z^{d_3} - Z^{d_3+1} \right)  \right\Vert_{\mathcal{U} (\mathcal{E}_{uc} (\delta_0))} \leq C r_2^{\sqrt{d_1+d_2 - d_3}}.$$
\end{theorem}

\begin{proof}
Let $r_1 = r_1 (\delta_0)$ be the constant of Corollary \ref{X,Y prod ineq final coro}.  We will prove that the inequality stated above holds for $r_2 = \max \lbrace r_1, \frac{1}{\sqrt{2}} \rbrace$.  

Fix $d_1, d_2, d_3 \in \mathbb{N}$ as above and denote $t = 2 \lfloor \frac{\sqrt{d_1 + d_2 -d_3}}{2} \rfloor$.  We note that $t \leq \sqrt{d_1 + d_2 -d_3} \leq t+2$ and thus it is enough to prove that there is a constant $C '$ such that 
$$\left\Vert X^{d_1} Y^{d_2} \left( Z^{d_3} - Z^{d_3+1} \right)  \right\Vert_{\mathcal{U} (\mathcal{E}_{uc} (\delta_0))} \leq C ' r_2^{t}.$$

Without loss of generality,  we can assume that $t \geq 4$ (for $t <4$,  the constant $C '$ can be chosen to be large enough such that $C ' r_2^{4} \geq 2$ and the needed inequality holds trivially).  

Denote
$$A_0 = \lbrace d_1-1 - k t : 0 \leq k \leq t-1 \rbrace.$$
Note that for every $0 \leq k \leq t-1 $
\begin{align*}
& d_1 - 1 \geq d_1-1 - k t > d_1-1 - t^2  +t  \geq^{t \leq \sqrt{d_1 + d_2 -d_3}} \\
& d_1-1 -(d_1+d_2-d_3) + t = d_3-d_2+t-1 \geq^{d_3 \geq d_2,  t \geq 4} 0,
\end{align*}
and thus $A_0 \subseteq \lbrace 0,...,d_1-1 \rbrace$.  

Also,  denote
$$B_0 = \bigcup_{k=0}^{t-1} \left\lbrace d_3-d_1+1 + k t + j : 0 \leq j \leq  \frac{t}{2}  -1 \right\rbrace,$$
(note that $t$ is always even and thus  $\frac{t}{2} \in \mathbb{N}$).   Note that for every $0 \leq k \leq t-1$ and every $0 \leq j \leq  \frac{t}{2}  -1$ it holds that
$$d_3-d_1+1 + k t + j \geq d_3-d_1 +1 \geq^{d_3 \geq d_1} 0$$
and
\begin{align*}
 d_3-d_1+1 + k t + j \leq d_3-d_1 +1 + t^2 - \frac{t}{2} \leq^{t \leq \sqrt{d_1+d_2-d_3}} d_3-d_1+1 + (d_1+d_2-d_3) - \frac{t}{2} = \\
  d_2+1- \frac{t}{2} \leq^{t \geq 4} d_2 -1.
\end{align*}

It follows that  $B_0 \subseteq \lbrace 0,...,d_2 -1 \rbrace$.

Thus,
\begin{align*}
& \left\Vert X^{d_1} Y^{d_2} \left( Z^{d_3} - Z^{d_3+1} \right)  \right\Vert_{\mathcal{U} (\mathcal{E}_{uc} (\delta_0))} \leq \\
& \left\Vert \left( \prod_{a \in \lbrace 0,...,d_1-1 \rbrace \setminus A_0} X_a \right) \right\Vert_{\mathcal{U} (\mathcal{E}_{uc} (\delta_0))} \left\Vert \left( \prod_{a \in A_0} X_a \right)  Y^{d_2} \left( Z^{d_3} - Z^{d_3+1} \right)  \right\Vert_{\mathcal{U} (\mathcal{E}_{uc} (\delta_0))}  \leq \\
& \left\Vert \left( \prod_{a \in A_0} X_a \right)  \left( \prod_{b \in B_0} Y_{b} \right) \left( Z^{d_3} - Z^{d_3+1} \right)  \right\Vert_{\mathcal{U} (\mathcal{E}_{uc} (\delta_0))}  \left\Vert \left( \prod_{b \in  \lbrace 0,...,d_2 -1 \rbrace \setminus B_0} Y_{b} \right) \right\Vert_{\mathcal{U} (\mathcal{E}_{uc} (\delta_0))}  \leq \\
& \left\Vert \left( \prod_{a \in A_0} X_a \right)  \left( \prod_{b \in B_0} Y_{b} \right) \left( Z^{d_3} - Z^{d_3+1} \right)  \right\Vert_{\mathcal{U} (\mathcal{E}_{uc} (\delta_0))}.
\end{align*}

It follows that it is enough to prove that there is a constant $C'$ such that 
\begin{equation*}
\left\Vert \left( \prod_{a \in A_0} X_a \right)  \left( \prod_{b \in B_0} Y_b \right) \left( Z^{d_3} - Z^{d_3+1} \right)  \right\Vert_{\mathcal{U} (\mathcal{E}_{uc} (\delta_0))} \leq C'  r_2^{t},
\end{equation*}
i.e.,  it is enough to prove that for every $(\pi, \B) \in \mathcal{E}_{uc} (\delta_0)$ and every $\xi \in \B$, it holds that
\begin{equation}
\label{reduction ineq}
\left\Vert  \pi \left( \left( \prod_{a \in A_0} X_a \right)  \left( \prod_{b \in B_0} Y_b \right) \left( Z^{d_3} - Z^{d_3+1} \right) \right) \xi  \right\Vert \leq C'  r_2^{t} \Vert \xi \Vert.
\end{equation}

Fix $(\pi, \B) \in \mathcal{E}_{uc} (\delta_0)$ and $\xi \in \B$.   For $1 \leq i \leq t  -1$,  define the sets
$$A_i = \lbrace d_1-1 - k t : i \leq k \leq t-1 \rbrace,$$
$$B_i = \bigcup_{k=i}^{t-1} \left\lbrace d_3-d_1+1 + k t + j : 0 \leq j \leq  \frac{t}{2}  -1 \right\rbrace.$$
Also define $A_t = B_t = \emptyset$.  For $0 \leq i \leq t$ denote
$$I_i = \left\Vert  \pi \left( \left( \prod_{a \in A_i} X_a \right)  \left( \prod_{b \in B_i} Y_b \right) \left( Z^{d_3} - Z^{d_3+1} \right) \right) \xi  \right\Vert.$$

We claim that in order to prove \eqref{reduction ineq}, it is sufficient to show that for every $0 \leq i \leq t-1$ it holds that 
\begin{equation}
\label{final reduction ineq}
I_{i} \leq 36  \left( \frac{1}{\sqrt{2}} \right)^{t} \Vert \xi \Vert +  r_1 I_{i+1},
\end{equation}
where $r_1$ is the constant of Corollary \ref{X,Y prod ineq final coro}.  Indeed,  if \eqref{final reduction ineq} holds,  then 
\begin{align*}
& I_0 \leq 36  \left( \frac{1}{\sqrt{2}} \right)^{t} \Vert \xi \Vert (1+r_1+...+r_1^{t-1}) + r_1^{t} I_t \leq \\
& \frac{36}{1-r_1}  \left( \frac{1}{\sqrt{2}} \right)^{t} \Vert \xi \Vert + r_1^{t} \Vert \pi (Z^{d_3} - Z^{d_3+1}) \xi \Vert \leq \\
& \frac{36}{1-r_1}  \left( \frac{1}{\sqrt{2}} \right)^{t} \Vert \xi \Vert + 2 r_1^{t} \Vert \xi \Vert \leq \left(\frac{36}{1-r_1} + 2 \right) r_2^{t},
\end{align*}
as needed.

We are left to prove \eqref{final reduction ineq}.  Fix $0 \leq i \leq t-1$.  For $0 \leq k \leq t-1$,  denote 
$$B_k ' =  \left\lbrace d_3-d_1+1 + k t + j : 0 \leq j \leq  \frac{t}{2}  -1 \right\rbrace,$$
thus $B_i = \bigcup_{k=i}^{t-1} B_k '$.  We note that 
$$d_3-\max A_{i+1} - \max B_i ' = d_3-(d_1-1- (i+1)t) - (d_3-d_1+it+\frac{t}{2}) = \frac{t}{2} +1.$$
Therefore,  by Theorem \ref{change of order thm},
\begin{align*}
\left\Vert \pi \left(  \left( \left( \prod_{a \in A_{i+1}} X_a \right)  \left( \prod_{b \in B_i '}  Y_{b} \right) -  \left( \prod_{b \in B_i '}  Y_{b} \right) \left( \prod_{a \in A_{i+1}} X_a \right) \right)   Z^{d_3}  \right) \right\Vert \leq \\ 
8 \left( \frac{1}{2} \right)^{d_3-\max A_{i+1} - \max B_i '}  = 
16  \left( \frac{1}{\sqrt{2}} \right)^{t}.
\end{align*}

Thus, 
\begin{align*}
\left\Vert \pi \left(  \left( \prod_{a \in A_{i}} X_a \right)  \left( \prod_{b \in B_i}  Y_{b} \right) Z^{d_3} -   \left( X_{d_1-1 - i t} \left( \prod_{b \in B_i '}  Y_{b}  \right) \right) \left( \prod_{a \in A_{i+1}} X_a \right)  \left( \prod_{b \in B_{i+1}} Y_b \right) Z^{d_3}  \right) \right\Vert \leq  \\
16  \left( \frac{1}{\sqrt{2}} \right)^{t}.
\end{align*}

By Observation \ref{product identity obs}, $Z^{d_3+1} = Z^{d_3} Z_{d_3}$ and thus 
$Z^{d_3}-Z^{d_3+1} = Z^{d_3} (e-Z_{d_3})$.  Using this and the inequality stated above,  we can deduce
\begin{dmath*}
I_i = \left\Vert  \pi \left( \left( \prod_{a \in A_i} X_a \right)  \left( \prod_{b \in B_i} Y_b \right) Z^{d_3} \left( e - Z_{d_3} \right) \right) \xi  \right\Vert \leq 
\left\Vert  \pi \left( \left( X_{d_1-1 - i t} \left( \prod_{b \in B_i '}  Y_{b}  \right) \right) \left( \prod_{a \in A_{i+1}} X_a \right)  \left( \prod_{b \in B_{i+1}} Y_b \right) Z^{d_3} (e-Z_{d_3})  \right) \xi  \right\Vert 
+ {16  \left( \frac{1}{\sqrt{2}} \right)^{t} \Vert \pi \left(e- Z_{d_3} \right) \xi \Vert} \leq 
\left\Vert  \pi \left( \left( X_{d_1-1 - i t} \left( \prod_{b \in B_i '}  Y_{b}  \right) (e-Z_{d_3}) \right) \left( \prod_{a \in A_{i+1}} X_a \right)  \left( \prod_{b \in B_{i+1}} Y_b \right) Z^{d_3} \right) \xi  \right\Vert + 32  \left( \frac{1}{\sqrt{2}} \right)^{t} \Vert \xi \Vert = \\
\left\Vert  \pi \left( X_{d_1-1 - i t} \left( \prod_{b \in B_i '}  Y_{b}  \right) (e-Z_{d_3}) \right) \pi \left( \left( \prod_{a \in A_{i+1}} X_a \right)  \left( \prod_{b \in B_{i+1}} Y_b \right) Z^{d_3} \right) \xi  \right\Vert + 32  \left( \frac{1}{\sqrt{2}} \right)^{t} \Vert \xi \Vert.
\end{dmath*}

Denote $\xi ' = \pi \left( \left( \prod_{a \in A_{i+1}} X_a \right)  \left( \prod_{b \in B_{i+1}} Y_b \right) Z^{d_3} \right) \xi $.  With this notation, we showed that
$$I_i \leq \left\Vert  \pi \left( X_{d_1-1 - i t} \left( \prod_{b \in B_i '}  Y_{b}  \right) (e-Z_{d_3}) \right) \xi ' \right\Vert + 32  \left( \frac{1}{\sqrt{2}} \right)^{t} \Vert \xi \Vert.$$
We note that 
$$\left\Vert \pi (e- Z_{d_3} ) \xi ' \right\Vert = \left\Vert  \pi \left( \left( \prod_{a \in A_{i+1}} X_a \right)  \left( \prod_{b \in B_{i+1}} Y_b \right) \left( Z^{d_3} - Z^{d_3+1} \right) \right) \xi  \right\Vert = I_{i+1}.$$

Thus,  in order to prove \eqref{final reduction ineq},  we are left to prove that 
$$\left\Vert  \pi \left( X_{d_1-1 - i t} \left( \prod_{b \in B_i '}  Y_{b}  \right) (e-Z_{d_3}) \right)  \xi ' \right\Vert \leq 4 \left( \frac{1}{\sqrt{2}} \right)^{t} \Vert \xi \Vert + r_1 \Vert \pi (e -Z_{d_3} ) \xi ' \Vert.$$
We recall that
\begin{align*}
\left\Vert  \pi \left( X_{d_1-1 - i t} \left( \prod_{b \in B_i '}  Y_{b}  \right) (e-Z_{d_3}) \right)  \xi ' \right\Vert = 
\left\Vert  \pi \left( X_{d_1-1 - i t} \left( \prod_{b=0}^{\frac{t}{2}-1}  Y_{(d_3-d_1+1+it)+b}  \right) (e-Z_{d_3}) \right)  \xi ' \right\Vert.
\end{align*}
Denoting $a_0 = d_1-1 - i t,  b_0 = d_3-d_1+1+it,  n= \frac{t}{2}$ and applying Corollary \ref{X,Y prod ineq final coro} (noting that $a_0 + b_0 = d_3$ and that we assumed that $n = \frac{t}{2} \geq 2$) yields that
\begin{align*}
\left\Vert  \pi \left( X_{d_1-1 - i t} \left( \prod_{b=0}^{\frac{t}{2}-1}  Y_{(d_3-d_1+1+it)+b}  \right) (e-Z_{d_3}) \right)  \xi ' \right\Vert \leq \max \left\lbrace r_1  \Vert \pi (e-Z_{d_3}) \xi ' \Vert, \frac{1}{2^{\frac{t}{2}-2}} \Vert \xi ' \Vert \right\rbrace \leq \\
4 \left( \frac{1}{\sqrt{2}} \right)^{t} \Vert \xi ' \Vert  + r_1 \Vert \pi (e-Z_{d_3}) \xi ' \Vert \leq 4 \left( \frac{1}{\sqrt{2}} \right)^{t} \Vert \xi \Vert  + r_1 \Vert \pi (e-Z_{d_3}) \xi ' \Vert,
\end{align*}
as needed.
\end{proof}

\begin{corollary}
\label{main ineq coro}
Let $\delta_0 : (0,2] \rightarrow (0,1]$ be a function.  There are constants $0 \leq r <1, C>0$ such that for every
 $d_1, d_2, d_3,d_4  \in \mathbb{N}$ such that  $d_1,  d_2 \geq \frac{1}{4} \min \lbrace d_3, d_4 \rbrace$ and $d_1 + d_2 - \max \lbrace d_3,d_4 \rbrace \geq \frac{1}{4} \min \lbrace d_3, d_4 \rbrace$,   it holds that 
$$\left\Vert X^{d_1} Y^{d_2}  \left(Z^{d_3} - Z^{d_4} \right) \right\Vert_{\mathcal{U} (\mathcal{E}_{uc} (\delta_0))} \leq \vert d_4-d_3 \vert  C r^{\sqrt{\min \lbrace d_3, d_4 \rbrace}},$$
and
$$\left\Vert Y^{d_1} X^{d_2} \left(Z^{d_3}  - Z^{d_4} \right) \right\Vert_{\mathcal{U} (\mathcal{E}_{uc} (\delta_0))} \leq \vert d_4-d_3 \vert  C r^{ \sqrt{\min \lbrace d_3, d_4 \rbrace}}.$$
\end{corollary}

\begin{proof}
Without loss of generality, we will assume that $d_4 >d_3$.  

We will start by proving the first inequality. 

Let $r_2, C$ be the constants of Theorem \ref{main ineq thm} and take $r= \sqrt{r_2}$.  We note that it is enough to prove that for any $0 \leq j \leq d_4 -d_3-1$ it holds that 
$$\left\Vert  X^{d_1} Y^{d_2}  \left(Z^{d_3+j} - Z^{d_3+j+1} \right)  \right\Vert_{\mathcal{U} (\mathcal{E}_{uc} (\delta_0))}  \leq C r^{\sqrt{d_3}}.$$
If $d_1, d_2 \leq d_3 +j$ this inequality follow immediately from Theorem \ref{main ineq thm}.  Otherwise,  either $d_2 > d_3+j$ or $d_1 > d_3 +j$ (or both) and thus 
$$\min \lbrace d_1, d_3 +j \rbrace + \min \lbrace d_2, d_3 +j \rbrace -d_3 -j \geq \frac{1}{4} d_3.$$
In this case,  we apply Theorem \ref{main ineq thm} replacing $d_i$ with $\min \lbrace d_i, d_3 +j \rbrace$ for $i=1,2$:
\begin{align*}
\left\Vert  X^{d_1} Y^{d_2}  \left(Z^{d_3+j} - Z^{d_3+j+1} \right)  \right\Vert_{\mathcal{U} (\mathcal{E}_{uc} (\delta_0))}  \leq \\
\left\Vert X^{\min \lbrace d_1, d_3 +j  \rbrace} Y^{\min \lbrace d_2, d_3 +j  \rbrace}  \left(Z^{d_3+j} - Z^{d_3+j+1} \right)  \right\Vert_{\mathcal{U} (\mathcal{E}_{uc} (\delta_0))}  \leq 
C r^{\sqrt{d_3}},
\end{align*}
as needed.

In order to prove the second inequality, we need to prove that for every $(\pi, \B) \in \mathcal{U} (\mathcal{E}_{uc} (\delta_0))$ it holds that 
$$\left\Vert \pi \left( Y^{d_2} X^{d_1} \left(Z^{d_3}  - Z^{d_4} \right)  \right) \right\Vert \leq (d_4-d_3) C r^{ \sqrt{d_3}}.$$

Fix $(\pi, \B) \in \mathcal{U} (\mathcal{E}_{uc} (\delta_0))$.  Let $\Phi : \rm H_3 (\mathbb{Z}) \rightarrow \rm H_3 (\mathbb{Z})$ be the isomorphism induced by 
$\Phi (x^{-1}) = y$ and $\Phi (y) = x$.   We note that 
$$\Phi (z) = \Phi (x^{-1} y^{-1} x y) = y x^{-1} y^{-1} x = [x,y^{-1}]^{-1} = z.$$

Extending $\Phi$ linearly yields that 
$$\Phi (Y^{d_2}) = X^{d_2},    \Phi (Z^{d_3}) = Z^{d_3},  \Phi (Z^{d_4}) = Z^{d_4},$$
and
$$\Phi (X^{d_1}) = \frac{1}{2^{d_1}} \sum_{b=-2^{d_1}+1}^0 y^b = y^{-2^{d_1}+1} \left( \frac{1}{2^{d_1}} \sum_{b=0}^{2^{d_1}-1} y^b \right)= y^{-2^{d_1}+1} Y^{d_1}.$$

Define $\pi_0 = \pi \circ \Phi$,  then $(\pi_0,  \B) \in  \mathcal{U} (\mathcal{E}_{uc} (\delta_0))$ and by the first inequality proven above it follows that 
\begin{align*}
C r^{\sqrt{d_3}} \geq \left\Vert \pi_0 \left( X^{d_1} Y^{d_2}  \left(Z^{d_3} - Z^{d_4} \right) \right)\right\Vert = \\
\left\Vert \pi \left( \Phi (X^{d_1}) \Phi (Y^{d_2})  \left(\Phi (Z^{d_3}) - \Phi (Z^{d_4}) \right) \right)\right\Vert = \\
\left\Vert \pi (y^{-2^{d_1}+1}) \pi \left( Y^{d_1} X^{d_2} \left(Z^{d_3} - Z^{d_4} \right) \right)\right\Vert = \\
\left\Vert \pi \left( Y^{d_1} X^{d_2} \left(Z^{d_3} - Z^{d_4} \right) \right)\right\Vert,
\end{align*}
as needed.
\end{proof}

\section{Relative Banach property (T) for $(\SL_3 (\mathbb{Z}),  \UT_3 (\mathbb{Z}))$ and $(\SL_3 (\mathbb{Z}), \LT_3 (\mathbb{Z}))$} 
\label{Relative Banach property sec}

In this section we will prove our main relative Banach property (T) result stated in the introduction.

For any $\lbrace i,j,k \rbrace = \lbrace 1,2,3 \rbrace$,  we denote 
$$H_{i,k} = \langle e_{i,j} (1), e_{j,k} (1),  e_{i,k} (1) \rangle < \SL_3 (\mathbb{Z}),$$
$$\widetilde{H}_{i,k} = \langle x_{i,j} (1), x_{j,k} (1),  x_{i,k} (1) \rangle < \St_3 (\mathbb{Z}),$$
For example,  $H_{1,3}$ is the group of uni-upper-triangular matrices $\UT_3 (\mathbb{Z})$ that appeared in the introduction. 

We will prove the following theorem:
\begin{theorem}
\label{relative prop T thm}
For any function $\delta_0 : (0, 2] \rightarrow (0,1]$ and any $1 \leq i,k \leq 3,  i \neq k$,  the pairs $(\SL_3 (\mathbb{Z}), H_{i,k})$ and $(\St_3 (\mathbb{Z}), \widetilde{H}_{i,k})$ have relative property $(T^\proj_{\mathcal{E}_{uc} (\delta_0)})$.  

In particular,  for any function $\delta_0 : (0, 2] \rightarrow (0,1]$,  the pairs $(\SL_3 (\mathbb{Z}), \UT_3 (\mathbb{Z}))$ and $(\SL_3 (\mathbb{Z}), \LT_3 (\mathbb{Z}))$ have relative property  $(T^\proj_{\mathcal{E}_{uc} (\delta_0)})$.  
\end{theorem}

Below, we will prove this theorem only for the pair $(\SL_3 (\mathbb{Z}), UT_3 (\mathbb{Z}))$. The proof will only use the Steinberg relations of $\SL_3 (\mathbb{Z})$ and thus it applies verbatim to the pair $(\St_3 (\mathbb{Z}), \widetilde{H}_{1,3} (\mathbb{Z}))$ (replacing each $e_{i,j}$ with $x_{i,j}$).   The proof for any other $H_{i,k}$ (or $\widetilde{H}_{i,k}$ in the case of the Steinberg group) follows from the case $UT_3 (\mathbb{Z})=H_{1,3}$ after permuting the indices.  

In order to prove this theorem,  we define the following: Let $1 \leq i,k \leq 3,  i \neq k$ and $d \in \mathbb{N}$.  Define $X_{i,k}^d \in \Prob_c (\SL_3 (\mathbb{Z}))$ by 
$$X_{i,k}^d = \frac{1}{2^{d}} \sum_{a=0}^{2^{d}-1} e_{i,k} (a).$$

Informally,  the idea of the proof is to look a product of all the $X_{i,k}^d$'s and to preform ``moves'' on this product with a small ``norm cost'' (such that this ``cost'' decrease as $d$ increases).  We will consider the following ``moves'' for $\lbrace i,j,k \rbrace = \lbrace 1,2,3 \rbrace$:
\begin{enumerate}
\item Switch moves on $H_{i,k}$: Replacing $X_{i,j}^{d_1} X_{i,k}^{d_3} X_{j,k}^{d_2}$ with  $X_{j,k}^{d_2}   X_{i,k}^{d_3} X_{i,j}^{d_1}$ and vice-versa.
\item Up/down moves on $H_{i,k}$: Replacing $X_{i,j}^{d_1} X_{i,k}^{d_3} X_{j,k}^{d_2}$ with $X_{j,k}^{d_2}   X_{i,k}^{d_4} X_{i,j}^{d_1}$ where $d_4 \in \mathbb{N}$ (when $d_4 > d_3$ this will be called an up move and when $d_4 < d_3$ this will be called a down move). 
\end{enumerate}

The following lemma bounds the ``norm cost'' of these moves:
\begin{lemma}
\label{norm cost bounds lemma}
Let  $\delta_0 : (0, 2] \rightarrow (0,1]$ and $\lbrace i,j,k \rbrace = \lbrace 1,2,3 \rbrace$.
For any $P_1,  P_1, T \in \Prob_c (\SL_3 (\mathbb{Z}))$ the following holds:
\begin{enumerate}
\item Switch moves on $H_{i,k}$ have small ``norm cost'': For any constants $d_1,d_2, d_3 \in \mathbb{N}$ with $d_1+d_2 \leq d_3 -2$ it holds that
\begin{align*}
& \left\Vert P_1 \left( X_{i,j}^{d_1} X_{i,k}^{d_3} X_{j,k}^{d_2}\right)  P_2 - T \right\Vert_{\mathcal{U} (\mathcal{E}_{uc} (\delta_0))} \leq \\
&\left\Vert P_1 \left( X_{j,k}^{d_2} X_{i,k}^{d_3}  X_{i,j}^{d_1} \right)  P_2 - T \right\Vert_{\mathcal{U} (\mathcal{E}_{uc} (\delta_0))} + 8 \left( \frac{1}{2} \right)^{d_3- (d_1+d_2)},
\end{align*}
\begin{align*}
& \left\Vert P_1 \left( X_{j,k}^{d_2} X_{i,k}^{d_3}  X_{i,j}^{d_1} \right) P_2 - T \right\Vert_{\mathcal{U} (\mathcal{E}_{uc} (\delta_0))} \leq \\
&\left\Vert P_1 \left(  X_{i,j}^{d_1} X_{i,k}^{d_3} X_{j,k}^{d_2}  \right) P_2 - T \right\Vert_{\mathcal{U} (\mathcal{E}_{uc} (\delta_0))} + 8 \left( \frac{1}{2} \right)^{d_3- (d_1+d_2)}.
\end{align*}

\item Up/down moves on $H_{i,k}$ have small ``norm cost'': Let $0 \leq r <1,  C >0$ be the constants of Corollary \ref{main ineq coro}.  For every  $d_1, d_2, d_3,d_4 \in \mathbb{N}$ such that $d_1,  d_2 \geq \frac{1}{4} \min \lbrace d_3, d_4 \rbrace$ and $d_1 + d_2 - \max \lbrace d_3,d_4 \rbrace \geq \frac{1}{4} \min \lbrace d_3, d_4 \rbrace$,  it holds that
\begin{align*}
& \left\Vert P_1 \left( X_{i,j}^{d_1} X_{i,k}^{d_3} X_{j,k}^{d_2} \right) P_2 - T \right\Vert_{\mathcal{U} (\mathcal{E}_{uc} (\delta_0))} \leq \\
& \left\Vert P_1 \left( X_{i,j}^{d_1} X_{i,k}^{d_4} X_{j,k}^{d_2}  \right) P_2 - T \right\Vert_{\mathcal{U} (\mathcal{E}_{uc} (\delta_0))} + \vert d_4 -d_3 \vert C r^{\sqrt{\min \lbrace d_3, d_4 \rbrace}},
\end{align*}
\begin{align*}
& \left\Vert P_1  \left( X_{j,k}^{d_2} X_{i,k}^{d_3} X_{i,j}^{d_1} \right)  P_2 - T \right\Vert_{\mathcal{U} (\mathcal{E}_{uc} (\delta_0))} \leq \\
& \left\Vert  P_1 \left( X_{j,k}^{d_2} X_{i,k}^{d_4} X_{i,j}^{d_1}  \right) P_2 - T \right\Vert_{\mathcal{U} (\mathcal{E}_{uc} (\delta_0))} +  \vert d_4 -d_3 \vert C r^{\sqrt{\min \lbrace d_3, d_4 \rbrace}}.
\end{align*}

\end{enumerate}
\end{lemma}

\begin{proof}
The bounds of the ``norm cost'' of switch moves follow directly from Theorem \ref{change of order thm}.  We will prove the bound of the first switch move (the proof of second bound is similar).  Let $d_1,d_2, d_3 \in \mathbb{N}$ with $d_1+d_2 \leq d_3 -2$,  then by Theorem \ref{change of order thm}
\begin{align*}
&\left\Vert P_1 \left( X_{i,j}^{d_1} X_{i,k}^{d_3} X_{j,k}^{d_2}\right)  P_2 - T \right\Vert_{\mathcal{U} (\mathcal{E}_{uc} (\delta_0))} \leq \\
& \left\Vert P_1 \left( X_{j,k}^{d_2} X_{i,k}^{d_3}  X_{i,j}^{d_1} \right)  P_2 - T \right\Vert_{\mathcal{U} (\mathcal{E}_{uc} (\delta_0))} +  \left\Vert P_1 \left(X_{i,j}^{d_1} X_{i,k}^{d_3} X_{j,k}^{d_2} -  X_{j,k}^{d_2} X_{i,k}^{d_3}  X_{i,j}^{d_1} \right)  P_2  \right\Vert_{\mathcal{U} (\mathcal{E}_{uc} (\delta_0))} \leq \\
 & \left\Vert P_1 \left( X_{j,k}^{d_2} X_{i,k}^{d_3}  X_{i,j}^{d_1} \right)  P_2 - T \right\Vert_{\mathcal{U} (\mathcal{E}_{uc} (\delta_0))} +  \left\Vert X_{i,j}^{d_1} X_{i,k}^{d_3} X_{j,k}^{d_2} -  X_{j,k}^{d_2} X_{i,k}^{d_3}  X_{i,j}^{d_1}   \right\Vert_{\mathcal{U} (\mathcal{E}_{uc} (\delta_0))} \leq \\
& \left\Vert P_1 \left( X_{j,k}^{d_2} X_{i,k}^{d_3}  X_{i,j}^{d_1} \right)  P_2 - T \right\Vert_{\mathcal{U} (\mathcal{E}_{uc} (\delta_0))} + 8 \left( \frac{1}{2} \right)^{d_3- (d_1+d_2)}.
\end{align*}

The bounds of the ``norm cost'' of up/down moves follow directly from Corollary \ref{main ineq coro}.  We will prove the bound of the first up/down move (the proof of second bound is similar). Let  $d_1, d_2, d_3,d_4 \in \mathbb{N}$ such that $d_1,  d_2 \geq \frac{1}{4} \min \lbrace d_3, d_4 \rbrace$ and $d_1 + d_2 - \max \lbrace d_3,d_4 \rbrace \geq \frac{1}{4} \min \lbrace d_3, d_4 \rbrace$.  By Corollary \ref{main ineq coro},
\begin{align*}
&\left\Vert P_1 \left( X_{i,j}^{d_1} X_{i,k}^{d_3} X_{j,k}^{d_2}\right)  P_2 - T \right\Vert_{\mathcal{U} (\mathcal{E}_{uc} (\delta_0))} \leq \\
& \left\Vert P_1 \left( X_{i,j}^{d_1} X_{i,k}^{d_4} X_{j,k}^{d_2} \right)  P_2 - T \right\Vert_{\mathcal{U} (\mathcal{E}_{uc} (\delta_0))} +  \left\Vert P_1 \left(X_{i,j}^{d_1} X_{i,k}^{d_3} X_{j,k}^{d_2} -  X_{i,j}^{d_1} X_{i,k}^{d_4} X_{j,k}^{d_2} \right)  P_2  \right\Vert_{\mathcal{U} (\mathcal{E}_{uc} (\delta_0))} \leq \\
 & \left\Vert P_1 \left( X_{i,j}^{d_1} X_{i,k}^{d_4} X_{j,k}^{d_2} \right)  P_2 - T  \right\Vert_{\mathcal{U} (\mathcal{E}_{uc} (\delta_0))} +  \left\Vert  X_{i,j}^{d_1} X_{i,k}^{d_3} X_{j,k}^{d_2} -  X_{i,j}^{d_1} X_{i,k}^{d_4} X_{j,k}^{d_2} \right\Vert_{\mathcal{U} (\mathcal{E}_{uc} (\delta_0))} \leq \\
& \left\Vert P_1 \left( X_{i,j}^{d_1} X_{i,k}^{d_4} X_{j,k}^{d_2} \right)  P_2 - T \right\Vert_{\mathcal{U} (\mathcal{E}_{uc} (\delta_0))} +  \vert d_4 -d_3 \vert C r^{\sqrt{\min \lbrace d_3, d_4 \rbrace}}.
\end{align*}
\end{proof}

For $d \in \mathbb{N}$,  we define $T_d,  S_d  \in \Prob_c (\SL_3 (\mathbb{Z}))$ as follows:
$$T_d = X_{1, 2}^{4d} X_{1, 3}^{10d} X_{2,  3}^{9d}  X_{2,  1}^{9d} X_{3, 1}^{10d} X_{3,  2}^{4d},$$
$$S_d = X_{2,  3}^{4d } X_{1,  3}^{10d} X_{1,  2}^{9d} X_{3,  2}^{9d}  X_{3,  1}^{10d}  X_{2 , 1}^{4d}.$$

Using the lemma above,  we will show that $\Vert T_d -S_{d \pm 1} \Vert$ is small:

\begin{lemma}
\label{step ineq lemma}
Let $\delta_0 : (0, 2] \rightarrow (0,1]$ be some function.  Then there are constants $0 \leq r <1$ and $L >0$ such that for every $d \in \mathbb{N}$  it holds that 
$$\left\Vert T_d  -S_{d \pm 1} \right\Vert_{\mathcal{U} (\mathcal{E}_{uc} (\delta_0))} \leq L d r^{\sqrt{d}}.$$
\end{lemma}

\begin{proof}
We will prove the bound only for $\Vert T_d - S_{d+1} \Vert$,   the proof for $\Vert T_d - S_{d-1} \Vert$ is similar.  Let $r, C$ be the constants of Corollary \ref{main ineq coro}.  Without loss of generality,  we can assume that $d >10$ (otherwise,  we can choose $L$ to be large enough such that $L r^{\sqrt{10}} \geq 2$ and the inequality holds trivially).  

The idea of the proof is to use switch moves and up/down moves to change $T_d$ into $S_{d+1}$ while book-keeping the ``norm cost'' using Lemma \ref{norm cost bounds lemma}.  We remark that when preforming up/down moves on $H_{i,k}$,  the order of the product of $X_{i,j}^{d_1},  X_{j,k}^{d_2},  X_{i,k}^{d_3}$ (where $\lbrace i,j,k \rbrace = \lbrace 1,2,3 \rbrace$) does not matter, since $X_{i,k}^{d_3}$ commutes with $X_{i,j}^{d_1}$ and $X_{j,k}^{d_2}$.  

\textbf{Step 1 (Up moves on $H_{1,3},  H_{3,1}$)}: We note that we assumed that $d >10$ and thus for $d_1 = 4d,  d_2 = 9d,  d_3 = 10d,  d_4 = 10 (d+1)$ the conditions of Lemma \ref{norm cost bounds lemma}(2) hold.  Preforming up moves on $H_{1,3}$ and $H_{3,1}$,  we get by Lemma \ref{norm cost bounds lemma} that
\begin{align*}
& \left\Vert T_d  -S_{d + 1} \right\Vert_{\mathcal{U} (\mathcal{E}_{uc} (\delta_0))} = \\
& \left\Vert \left( X_{1, 2}^{4d} X_{1, 3}^{10d} X_{2,  3}^{9d}  \right) \left( X_{2,  1}^{9d} X_{3, 1}^{10d} X_{3,  2}^{4d} \right) -S_{d + 1} \right\Vert_{\mathcal{U} (\mathcal{E}_{uc} (\delta_0))} \leq \\
& \left\Vert \left( X_{1, 2}^{4d} X_{1, 3}^{10(d+1)} X_{2,  3}^{9d}  \right) \left( X_{2,  1}^{9d} X_{3, 1}^{10(d+1)} X_{3,  2}^{4d} \right) -S_{d + 1} \right\Vert_{\mathcal{U} (\mathcal{E}_{uc} (\delta_0))} + 20 C r^{\sqrt{10d}}.
\end{align*}

Bottom line: The ``norm cost'' of this step is $20 C r^{\sqrt{10d}} \leq 20 C d r^{\sqrt{d}}$.  \\

\textbf{Step 2 (Down moves on $H_{2,3}, H_{2,1}$):} For $d_1 = 10 (d+1),  d_2 = 9d,  d_3 = 9d,  d_4 = 4(d+1)$,  the conditions of Lemma \ref{norm cost bounds lemma}(2) are fulfilled and thus we can preform the following down move on $H_{2,3}$:
\begin{align*}
& \left\Vert  X_{1, 2}^{4d} \left( X_{1, 3}^{10(d+1)} X_{2,  3}^{9d} X_{2,  1}^{9d} \right) X_{3, 1}^{10(d+1)} X_{3,  2}^{4d} -S_{d + 1} \right\Vert_{\mathcal{U} (\mathcal{E}_{uc} (\delta_0))} \leq \\
& \left\Vert  X_{1, 2}^{4d} \left( X_{1, 3}^{10(d+1)} X_{2,  3}^{4 (d+1)} X_{2,  1}^{9d} \right) X_{3, 1}^{10(d+1)} X_{3,  2}^{4d} -S_{d + 1} \right\Vert_{\mathcal{U} (\mathcal{E}_{uc} (\delta_0))} + (9d- 4 (d+1)) C r^{\sqrt{4(d+1)}}.
\end{align*}
After that,  we preform a down move on $H_{2,1}$  with $d_1 = 4(d+1), d_2 = 10 (d+1), d_3 = 9d,  d_4 = 4(d+1)$ (using Lemma \ref{norm cost bounds lemma}(2) again):
\begin{align*}
& \left\Vert  X_{1, 2}^{4d} X_{1, 3}^{10(d+1)} \left( X_{2,  3}^{4 (d+1)} X_{2,  1}^{9d}  X_{3, 1}^{10(d+1)} \right) X_{3,  2}^{4d} -S_{d + 1} \right\Vert_{\mathcal{U} (\mathcal{E}_{uc} (\delta_0))}  \leq \\
& \left\Vert  X_{1, 2}^{4d} X_{1, 3}^{10(d+1)} \left( X_{2,  3}^{4 (d+1)} X_{2,  1}^{4(d+1)}  X_{3, 1}^{10(d+1)} \right) X_{3,  2}^{4d} -S_{d + 1} \right\Vert_{\mathcal{U} (\mathcal{E}_{uc} (\delta_0))} +  (9d- 4 (d+1)) C r^{\sqrt{4(d+1)}}.
\end{align*}

Bottom line: The ``norm cost'' of this step is $2 (9d- 4 (d+1)) C r^{\sqrt{4(d+1)}}  \leq 10 d C r^{\sqrt{d}}$.  \\

\textbf{Step 3 (Switch moves on $H_{1,3},  H_{3,1}$):} Preforming switch moves on $H_{1,3}$ and $H_{3,1}$ with $d_1 = 4d,  d_2 = 4(d+1),  d_3 = 10 (d+1)$,  we get by Lemma \ref{norm cost bounds lemma}(1) that
\begin{align*}
& \left\Vert  \left( X_{1, 2}^{4d} X_{1, 3}^{10(d+1)} X_{2,  3}^{4 (d+1)} \right) \left( X_{2,  1}^{4(d+1)}  X_{3, 1}^{10(d+1)} X_{3,  2}^{4d} \right) -S_{d + 1} \right\Vert_{\mathcal{U} (\mathcal{E}_{uc} (\delta_0))} \leq \\
& \left\Vert  \left(  X_{2,  3}^{4 (d+1)} X_{1, 3}^{10(d+1)} X_{1, 2}^{4d} \right) \left( X_{3,  2}^{4d}   X_{3, 1}^{10(d+1)} X_{2,  1}^{4(d+1)} \right) -S_{d + 1} \right\Vert_{\mathcal{U} (\mathcal{E}_{uc} (\delta_0))} + 16 \left(\frac{1}{2} \right)^{2 (d+1)}.
\end{align*}

Bottom line: The ``norm cost'' of this step is $16 \left(\frac{1}{2} \right)^{2 (d+1)} \leq 16 d r^{\sqrt{d}}$ (by the choice of $r_2$ in the proof of Theorem \ref{main ineq thm} it follows that $r \geq \frac{1}{2}$).  \\

\textbf{Step 4 (Up moves on $H_{1,2},  H_{3,2}$):} For $d_1 = 10 (d+1),  d_2 = 4d,  d_3 = 4d,  d_4 = 9 (d+1)$,  the conditions of Lemma \ref{norm cost bounds lemma}(2) are fulfilled and thus we can preform the following up move on $H_{1,2}$: 
\begin{align*}
& \left\Vert  X_{2,  3}^{4 (d+1)} \left( X_{1, 3}^{10(d+1)} X_{1, 2}^{4d}  X_{3,  2}^{4d}  \right)   X_{3, 1}^{10(d+1)} X_{2,  1}^{4(d+1)} -S_{d + 1} \right\Vert_{\mathcal{U} (\mathcal{E}_{uc} (\delta_0))}  \leq \\
& \left\Vert  X_{2,  3}^{4 (d+1)} \left( X_{1, 3}^{10(d+1)} X_{1, 2}^{9(d+1)}  X_{3,  2}^{4d}  \right)   X_{3, 1}^{10(d+1)} X_{2,  1}^{4(d+1)} -S_{d + 1} \right\Vert_{\mathcal{U} (\mathcal{E}_{uc} (\delta_0))} + (9 (d+1)-4d) C r^{\sqrt{4d}}.
\end{align*}
After that,  we preform an up move on $H_{3,2}$  with $d_1 = 9(d+1), d_2 = 10 (d+1), d_3 = 4d,  d_4 = 9(d+1)$ (using Lemma \ref{norm cost bounds lemma}(2) again):
\begin{align*}
& \left\Vert  X_{2,  3}^{4 (d+1)} X_{1, 3}^{10(d+1)} \left(  X_{1, 2}^{9(d+1)}  X_{3,  2}^{4d}   X_{3, 1}^{10(d+1)} \right)  X_{2,  1}^{4(d+1)} -S_{d + 1} \right\Vert_{\mathcal{U} (\mathcal{E}_{uc} (\delta_0))}  \leq \\
& \left\Vert  X_{2,  3}^{4 (d+1)} X_{1, 3}^{10(d+1)} \left(  X_{1, 2}^{9(d+1)}  X_{3,  2}^{9(d+1)}   X_{3, 1}^{10(d+1)} \right)  X_{2,  1}^{4(d+1)} -S_{d + 1} \right\Vert_{\mathcal{U} (\mathcal{E}_{uc} (\delta_0))}  + (9 (d+1)-4d) C r^{\sqrt{4d}} = \\
& \left\Vert  S_{d+1}  -S_{d + 1} \right\Vert_{\mathcal{U} (\mathcal{E}_{uc} (\delta_0))}  + (9 (d+1)-4d) C r^{\sqrt{4d}} = (9 (d+1)-4d) C r^{\sqrt{4d}}.
\end{align*}

Bottom line: The ``norm cost'' of this step is $2 (9 (d+1)-4d) C r^{\sqrt{4d}} \leq 12 d C r^{\sqrt{d}}$.  \\

Using the bounds of the norm costs at each step, we deduce that 
$$ \left\Vert T_d  -S_{d + 1} \right\Vert_{\mathcal{U} (\mathcal{E}_{uc} (\delta_0))} \leq 20 C d r^{\sqrt{d}} + 10 d C r^{\sqrt{d}}+ 16 d r^{\sqrt{d}} + 12 d C r^{\sqrt{d}} = (42C + 16) d r^{\sqrt{d}},$$
as needed.
\end{proof}

After this,  we can prove Theorem \ref{relative prop T thm}:
\begin{proof}
Fix $\delta_0 : (0,2] \rightarrow (0,1]$. 

Define $h_d \in \Prob_c (\SL_3 (\mathbb{Z}))$ by 
$$h_d = 
\begin{cases}
T_d & d \text{ is odd} \\
S_d & d \text{ is even}
\end{cases}.$$

By Lemma \ref{step ineq lemma},  there are $L> 0,  0 \leq r <1$ such that for every $d$,
$$\Vert h_d - h_{d+1} \Vert_{\mathcal{U} (\mathcal{E}_{uc} (\delta_0))} \leq d L r^{\sqrt{d}}.$$ 
Thus $\lbrace h_d \rbrace_{d \in \mathbb{N}}$ is a Cauchy sequence with respect to $\Vert . \Vert_{\mathcal{U} (\mathcal{E}_{uc} (\delta_0))}$ and it has a limit that we will denote $f \in C_{\mathcal{U} (\mathcal{E}_{uc} (\delta_0))}$.  We note that for every odd $d$, 
\begin{dmath*}
\Vert (e-e_{1,2} (1)) h_d \Vert_{\mathcal{U} (\mathcal{E}_{uc} (\delta_0))} \leq 
\left\Vert (e-e_{1,2} (1)) X_{1, 2}^{4d} \right\Vert_{\mathcal{U} (\mathcal{E}_{uc} (\delta_0))}  \leq \\
\left\Vert \frac{1}{2^{4d}} (e- e_{1,2} (2^{4d}))  \right\Vert_{\mathcal{U} (\mathcal{E}_{uc} (\delta_0))}  \leq  
\frac{1}{2^{4d-1}}.
\end{dmath*}

Therefore $\Vert (e-e_{1,2} (1)) f \Vert_{\mathcal{U} (\mathcal{E}_{uc} (\delta_0))} =0$. This implies that for every $(\pi, \B) \in \mathcal{U} (\mathcal{E}_{uc} (\delta_0))$,  $\pi (e_{1,2} (1) f) =  \pi (f)$  and thus $\im (\pi (f)) \subseteq \B^{\pi (\langle e_{1,2} (1) \rangle)}$.  Similarly, for every even $d$, 
\begin{dmath*}
\Vert (e-e_{2,3} (1)) h_d \Vert_{\mathcal{U} (\mathcal{E}_{uc} (\delta_0))} \leq 
\left\Vert (e-e_{1,2} (1)) X_{2,3}^{4d} \right\Vert_{\mathcal{U} (\mathcal{E}_{uc} (\delta_0))}  \leq \\
\left\Vert \frac{1}{2^{4n}} (e- e_{1,2} (2^{4d}))  \right\Vert_{\mathcal{U} (\mathcal{E}_{uc} (\delta_0))}  \leq  
\frac{1}{2^{4d-1}},
\end{dmath*}
and thus for every $(\pi, \B) \in \mathcal{U} (\mathcal{E}_{uc} (\delta_0))$,  $\im (\pi (f)) \subseteq \B^{\pi (\langle e_{2,3} (1) \rangle)}$.  It follows that for every $(\pi, \B) \in \mathcal{U} (\mathcal{E}_{uc} (\delta_0))$, 
$\im (\pi (f)) \subseteq \B^{\pi (\langle e_{1,2} (1), e_{2,3} (1) \rangle)} = \B^{\pi (\UT_3 (\mathbb{Z}))}$ as needed.
\end{proof}

As a corollary,  we get Theorem \ref{relative prop T - intro thm} that appeared in the introduction:
\begin{corollary}
\label{relative prop T for UT coro}
The pairs $(\SL_3 (\mathbb{Z}), \UT_3 (\mathbb{Z}))$ and $(\SL_3 (\mathbb{Z}), \LT_3 (\mathbb{Z}))$ have relative property $(T_{\B})$ for every uniformly convex Banach space $\B$.  
\end{corollary}

\begin{proof}
The proof readily follows from Theorem \ref{relative prop T thm} and Proposition \ref{relative T^proj implies relative T prop}.
\end{proof}

\section{Banach property (T) for $\SL_3 (\mathbb{Z})$}
\label{SL sec}

In this section,  we will prove our main result regarding the Banach property (T) of $\SL_3 (\mathbb{Z})$.

We fix the following terminology:  an elementary subgroup of $\SL_n (\mathbb{Z})$ is a subgroup of the form $E_{i,j} = \lbrace e_{i,j} (a) : a \in \mathbb{Z} \rbrace$ for some $1 \leq i,j \leq n,  i \neq j$.  A theorem of Carter and Keller is that these subgroups boundedly generate $\SL_n (\mathbb{Z})$:
\begin{theorem} \cite[Main Theorem]{CK}
\label{bounded gene result thm}
Let $n \geq 3$.  The group $\SL_n (\mathbb{Z})$ is boundedly generated by all the elementary subgroups. 
\end{theorem}

This allows us to prove the following theorem:
\begin{theorem}
\label{T for SL_3 Z  thm}
For every super-reflexive Banach space $\B$,  the group $\SL_3 (\mathbb{Z})$ has property $(T_\B)$.
\end{theorem}

\begin{proof}
As noted above in \cref{Banach property $(T)$ for super-reflexive Banach spaces},   it is enough to prove the result for uniformly convex Banach spaces.  Let $\B$ be some uniformly convex Banach space. 

Denote $\UT_3 (\mathbb{Z})$ and $\LT_3 (\mathbb{Z})$ be the subgroups of uni-upper-triangular and uni-lower-triangular matrices defined above.

By Theorem \ref{bounded gene result thm},  $\UT_3 (\mathbb{Z})$ and $\LT_3 (\mathbb{Z})$ boundedly generate $\SL_3 (\mathbb{Z})$ and by Corollary \ref{relative prop T for UT coro},  $(\SL_3 (\mathbb{Z}), \UT_3 (\mathbb{Z}))$ and  $(\SL_3 (\mathbb{Z}), \LT_3 (\mathbb{Z}))$ both have relative property $(T_\B)$. Thus, by Theorem \ref{bounded generation imply prop T thm},   $\SL_3 (\mathbb{Z})$ has property $(T_{\B})$. 

\end{proof}

\begin{remark}
At this point,  it is also possible to give a bounded generation proof that shows that for that for every super-reflexive Banach space $\B$ and every $n \geq 3$,  the group $\SL_n (\mathbb{Z})$ has property $(T_\B)$ via induction on $n$ (using the $\SL_3 (\mathbb{Z})$ case as the basis of the induction).  We will not do this here,  since this result will follow from our general treatment of Banach property $(T)$ for simple Lie groups below. 
\end{remark}

\section{Banach property $(T)$ for simple Lie groups}
\label{Banach property $(T)$ for simple Lie groups}

The aim of the section is to prove Theorem \ref{H rank intro thm} that appeared in the introduction.  We start by stating Howe-Moore's Theorem for reflexive  Banach spaces:
\begin{theorem}
{[Howe-Moore's Theorem for reflexive  Banach spaces \cite{Veech}]}
Let $\B$ be a reflexive Banach space and $G$ be a connected simple real Lie group with a finite center.  Then for every continuous linear isometric representation $\pi : G \rightarrow O (\B)$ such that $\B^{\pi (G)} = \lbrace 0 \rbrace$ it holds for every $\xi \in \B$ and $\eta \in \B^*$ that 
$$\lim_{g \rightarrow \infty} \langle \pi  (g) \xi , \eta \rangle =0.$$
\end{theorem}

\begin{corollary}
\label{same B decomp coro}
Let $\B$ be a uniformly convex Banach space and $G$ be a connected simple real Lie group with a finite center.   For every unbounded subgroup $H < G$ and every continuous linear isometric representation $\pi :  G \rightarrow O (\B)$ it holds that $\B^{\pi (G))} =  \B^{\pi ( H)}$ and that $\B ' (\pi) = \B ' (\left. \pi \right\vert_H)$. 
\end{corollary}

\begin{proof}
Let $\xi \in  \B^{\pi ( H)} \cap \B ' (\pi)$ and denote $\pi '$ the restriction of $\pi$ to $\B ' (\pi)$.  Fix $h_n \in H$ tending to infinity.  By Howe-Moore it follows for every $\eta \in \B^*$ that
$$\langle \xi, \eta \rangle = \lim_n \langle \pi (h_n) \xi,  \eta \rangle =0,$$
thus $\xi = 0$.  This shows that $\B^{\pi ( G)} =  \B^{\pi ( H)}$.  Similarly,  $\B^{\pi^* ( G)} =  \B^{\pi^* ( H)}$. Recall that $\B '$ and  $\B ' (\left. \pi \right\vert_H)$ are the annihilators of  $\B^{\pi^* ( G)}$ and $\B^{\pi^* ( H)}$,  and thus $\B ' (\pi) = \B ' (\left. \pi \right\vert_H)$.  
\end{proof}

\begin{corollary}
\label{inherited prop T for Lie coro}
Let $\B$ be a uniformly convex Banach space and $G$ be a connected simple real Lie group with a finite center.  For every unbounded subgroup $H < G$,  if $H$ has property $(T_{\B})$, then so does $G$.
\end{corollary}

\begin{proof}
By Observation \ref{B^pi(G)=0 obs}, it is enough to show that every continuous isometric representation $\pi : G \rightarrow O( \B)$ with $\B^{\pi (G)} = \lbrace 0 \rbrace$ does not have almost invariant vectors.  Fix $\pi : G \rightarrow O( \B)$ with $\B^{\pi (G)} = \lbrace 0 \rbrace$.

By Corollary \ref{same B decomp coro},  it holds that $\B^{\pi ( H)}  = \lbrace 0 \rbrace$.  Thus it follows that $\left. \pi \right\vert_H$ does not have ($H$-)almost invariant vectors and as a result $\pi$ does not have ($G$-)almost invariant vectors. 
\end{proof}

An immediate consequence of this corollary is the following:
\begin{corollary}
\label{Banach prop T for SL_3 R coro}
The group $\SL_3 (\mathbb{R})$ has property $(T_\B)$ for uniformly convex Banach space $\B$.
\end{corollary}

\begin{proof}
For every uniformly convex Banach space $\B$,  combining Theorem \ref{T for SL_3 Z  thm} and Corollary \ref{inherited prop T for Lie coro} yields that $\SL_3 (\mathbb{R})$ has property $(T_\B)$.
\end{proof}

A standard argument allows us to use this corollary in order to prove Theorem \ref{H rank intro thm} that appeared in the introduction :
\begin{theorem}
Let $G$ be a connected simple real Lie group with a finite center and $\mathfrak{g}$ the Lie algebra of $G$.  If  $\mathfrak{g}$ contains $\mathfrak{sl}_3 (\mathbb{R})$ as a Lie sub-algebra, then $G$ all its lattices have property $(T_{\B})$ for every uniformly convex Banach space $\B$.  
\end{theorem}

\begin{proof}
We note since $Z(G)$ is finite,  it follows that if $\rm Ad (G) = G / Z(G)$ has property $(T_{\B})$ for every uniformly convex Banach space $\B$,  then $G$ has  property $(T_{\B})$ for every uniformly convex Banach space $\B$.  Thus we can replace $G$ with $\rm Ad (G) = G / Z(G)$ and assume that $G$ is algebraic. 

Fix a uniformly convex Banach space $\B$.  By Corollary \ref{inherited prop T for Lie coro},  in order to prove that $G$ has property $(T_\B)$ it is enough to show that there is an unbounded subgroup $G$ such that has property $(T_{\B})$. 

By our assumption on $\mathfrak{g}$, the group $G$ contains a subgroup $H$ whose (algebraic) simply connected covering (see \cite[Definition 1.4.12]{Mar}) is isomorphic to $\SL_3 (\mathbb{R})$.  Thus there is $\varphi : \SL_3 (\mathbb{R}) \rightarrow H$ such that the cokernel of $\varphi$ is finite.  Since property $(T_{\B})$ is inherited by quotients,  it follows from Corollary \ref{Banach prop T for SL_3 R coro} that $\varphi (H)$ has property $(T_{\B})$ as needed.

Last,  note that by Corollary \ref{passing to lattices coro},  if $G$ has property $(T_\B)$ for every uniformly convex Banach space $\B$,  then every lattice $\Gamma < G$ has property $(T_\B)$ for every uniformly convex Banach space $\B$.
\end{proof}

\section{Applications}
\label{app sec}

\subsection{Banach fixed point property for simple Lie groups}
\label{subsec alg groups}

Let $\B$ be a Banach space and $G$ be a topological group.  An affine isometric action of $G$ on $\B$ is a continuous homomorphism $\rho : G \rightarrow \Isom_{aff} (\B)$, where $\Isom_{aff} (\B)$ denotes the group of affine isometric automorphisms of $\B$.  We recall such that $\rho$ is of the form 
$$\rho (g) \xi = \pi  (g) \xi + c(g),  \forall \xi \in \B$$
where $\pi : G  \rightarrow O(\B)$ is a continuous isometric linear representation that is called the linear part of $\rho$ and $c: G \rightarrow \B$ is a continuous $1$-cocycle into $\pi$, i.e.,  it is a continuous map such that for every $g,h \in G$, 
$$c (gh) = c (g) + \pi (g) c (h).$$

The group $G$ is said to have \textit{property $(F_{\B})$} if every affine isometric action of $G$ on $\B$ admits a fixed point.  

Below, we will prove Theorem \ref{H rank intro thm 2} that states that a simple Lie group whose Lie algebra contains $\mathfrak{sl}_4 (\mathbb{R})$ has property $(F_{\B})$ for every super-reflexive Banach space.  For this we will state the following results: 
\begin{theorem}
{Metric Mautner phenomenon,  \cite[Theorem 1.3]{BG}}
\label{metric maut thm}
Let $G$ be a connected simple real Lie group with a finite center.  Assume that $G$ acts isometrically and continuously on a metric space $(\mathbb{X},d)$.  If there is a non-compact point stabilizer (of the action of $G$ on $\mathbb{X}$),  then the action of $G$ on  $\mathbb{X}$ admits a fixed point. 
\end{theorem}

\begin{lemma}\cite[Lemma 5.6]{BFGM}
\label{H1 times H2 lemma}
Let $H_1 \times H_2$ be a topological group that acts by a continuous affine isometric action $\rho$ on a uniformly convex Banach space $\B$.  If the linear part of $\rho$ restricted to $H_1$ does not have almost invariant vectors, then the action of $H_2$ on $\B$ admits a fixed point.
\end{lemma}

Last,  using \cite{BFGM} we can show that for higher rank Lie groups property $(F_\B)$ for every uniformly convex $\B$ is inherited to (and from) passing to lattices:
\begin{proposition}
\label{prop F_B passes to lattices}
Let $G$ be a connected higher rank simple real Lie group and $\Gamma < G$ a lattice.  The group $G$ has property $(F_\B)$ for every uniformly convex Banach space if and only if the group $\Gamma$ has property $(F_\B)$ for every uniformly convex Banach space $\B$.
\end{proposition}

\begin{proof}
In \cite[Proposition 8.8]{BFGM} it was shown that for any locally compact group $G$,  any lattice $\Gamma < G$ and any uniformly convex Banach space $\B$ the following holds:
\begin{itemize}
\item If $\Gamma$ has property $(F_\B)$,  then $G$ has property $(F_\B)$.
\item If $G$ has property $(F_{L^2 (G / \Gamma ; \B)})$ and $\Gamma$ is $2$-integrable (see \cite[Definition 8.2]{BFGM}), then $\Gamma$ has property $(F_\B)$.
\end{itemize}
Thus it readily follows that if $\Gamma$ has property $(F_\B)$ for every uniformly convex Banach space $\B$, then so does $G$. 

In the other direction,  we use the following two facts: First,  Shalom \cite{Shalom3} showed that every lattice in a connected higher rank simple Lie groups in $2$-integrable.  Second,  by Theorem \ref{L^2 sum of uc spaces thm},  $L^2 (G / \Gamma ; \B)$ is uniformly convex.  
\end{proof}

Our starting point towards proving Theorem \ref{H rank intro thm 2} it to prove $\SL_4 (\mathbb{R})$ has property $(F_\B)$ for every super-reflexive Banach space $\B$:

\begin{theorem}
\label{FP for SL_4 thm}
For every super-reflexive Banach space $\B$,  the group $\SL_{4} (\mathbb{R})$ has property $(F_{\B})$.  
\end{theorem}

\begin{proof}
By \cite[Proposition 2.13]{BFGM} it is enough to consider uniformly convex Banach spaces.   


Fix some uniformly convex Banach space $\B$.  

Let $\rho: \SL_{4} (\mathbb{R})  \rightarrow \rm Isom_{aff} (\B)$ be a continuous affine isometric action of $\SL_4 (\mathbb{R})$ on $\B$ with a linear part $\pi$.  As in \cite[5.c]{BFGM},  the fact that $\SL_4 (\mathbb{R})$ has a compact abelianization allows us to assume that $\B^{\pi (\SL_4 (\mathbb{R}))} = \lbrace 0 \rbrace$.   

Define $H < \SL_{4} (\mathbb{R})$ to be the subgroup
$$H = \left\lbrace \left( \begin{matrix}
 & 0 \\
A & 0 \\
 & 0 \\
 0 \hspace*{0.1in} 0 \hspace*{0.1in}  0  & \frac{1}{\rm det (A)}
\end{matrix} \right) : A \in \rm GL_{3} (\mathbb{R}) \right\rbrace.$$

We note that $H \cong \rm GL_{3} (\mathbb{R}) \cong \SL_{3} (\mathbb{R}) \times \mathbb{R}^{*}$.  By Corollary \ref{same B decomp coro},  $\B^{\pi (\SL_{3} (\mathbb{R}))} = \lbrace 0 \rbrace$.  Thus, by Corollary \ref{Banach prop T for SL_3 R coro},  the representation $\left.  \pi \right\vert_{\SL_{3} (\mathbb{R})}$ does not have almost invariant vectors and by Lemma \ref{H1 times H2 lemma},  the $\rho$-action of $\mathbb{R}^{*}$ on $\B$ has a fixed point.  It follows from Theorem \ref{metric maut thm} that the $\rho$-action of $\SL_4 (\mathbb{R})$ on $\B$ has a fixed point as needed.
\end{proof}

\begin{theorem}
Let $G$ be a connected simple real Lie group with a finite center and $\mathfrak{g}$ be the Lie algebra of $G$.   If  $\mathfrak{g}$ contains $\mathfrak{sl}_4 (\mathbb{R})$ as a Lie sub-algebra, then $G$ and any lattice $\Gamma < G$ have property $(F_{\B})$ for every super-reflexive Banach space $\B$.  
\end{theorem}

\begin{proof}
By \cite[Proposition 2.13]{BFGM} it is enough to consider uniformly convex Banach spaces.  

We note that since $Z(G)$ is finite,  if $\rm Ad (G) = G / Z(G)$ has property $(F_{\B})$ for every uniformly convex Banach space $\B$,  then $G$ has property $(F_{\B})$ for every uniformly convex Banach space $\B$. Thus we can replace $G$ with $\rm Ad (G) = G / Z(G)$ and assume that $G$ is algebraic. 

By Proposition \ref{prop F_B passes to lattices},  if $G$ has property $(F_{\B})$ for every uniformly convex Banach space $\B$,   then every lattice of $G$ has property $(F_{\B})$ for every uniformly convex Banach space $\B$.  Thus,  it is enough to prove that $G$ has property $(F_{\B})$ for every uniformly convex Banach space $\B$.  

Fix a uniformly convex Banach space $\B$ and a continuous affine isometric action of $G$ on $\B$.  

By our assumption on $\mathfrak{g}$,  the group $G$ contains a subgroup $H$ whose (algebraic) simply connected covering is isomorphic to $\SL_4 (\mathbb{R})$.  Thus there is $\varphi : \SL_4 (\mathbb{R}) \rightarrow H$ such that the cokernel of $\varphi$ is finite.   Observe that property $(F_\B)$ is preserved under passing to quotients and thus by Theorem \ref{FP for SL_4 thm} it follows that $\varphi (\SL_4 (\mathbb{R}))$ has property $(F_\B)$.  It follows that the action of $\varphi (\SL_4 (\mathbb{R}))$ on $\B$ has a fixed point,  so there is a non-compact point stabilizer of the action of $G$ on $\B$.  By Theorem \ref{metric maut thm} it follows that the $G$ action on $\B$ has a fixed point.
\end{proof}

\subsection{Super-expanders}
\label{Super-expanders subsec}
We start by recalling the definition of Mendel and Naor \cite{MendelNaor} for super-expanders. 

Let $\B$ be a Banach space and $\lbrace (V_i,E_i) \rbrace_{i \in \mathbb{N}}$ be a sequence of finite graphs with uniformly bounded degree,  such that $\lim_i \vert V_i \vert = \infty$.  We say that $\lbrace (V_i,E_i) \rbrace_{i \in \mathbb{N}}$ has a \textit{Poincar\'{e} inequality with respect to $\B$} if there are constants $p, \gamma \in (0, \infty)$ such that for every $i \in \mathbb{N}$ and every $\phi : V_i \rightarrow \B$ we have
$$\frac{1}{\vert V_i \vert^2} \sum_{(u,v) \in V_i \times V_i} \Vert \phi (u) - \phi (v) \Vert^p \leq \frac{\gamma}{\vert V_i \vert} \sum_{(x,y) \in E_i} \Vert \phi (x) - \phi (y) \Vert^p.$$

The sequence $\lbrace (V_i,E_i) \rbrace_{i \in \mathbb{N}}$ is called a \textit{super-expander family} if it has a Poincar\'{e} inequality with respect to every super-reflexive Banach space (or equivalently for every uniformly convex Banach space). 

For Cayley graphs, the following proposition of Lafforgue gives a relation between Poincar\'{e} inequality of Cayley graphs and Banach property $(T)$:
\begin{proposition}\cite[Proposition 5.2]{Laff1}
\label{prop T imply exp prop}
Let $\Gamma$ be a finitely generated discrete group and let $\lbrace N_i \rbrace_{i \in \mathbb{N}}$ be a sequence of finite index normal subgroups of $\Gamma$ such that $\bigcap_i N_i = \lbrace 1 \rbrace$.  Also let $\B$ be a Banach space.  If $\Gamma$ has Banach property $(T_{\B '})$ for  $\B ' = \ell^2 (\bigcup_i G/N_i; \B)$,  then for every fixed finite symmetric generating set $S$, the family of Cayley graphs of $\lbrace (G/N_i,S /N_i) \rbrace_{i \in \mathbb{N}}$ has a Poincar\'{e} inequality with respect to $\B$.
\end{proposition}

A consequence of this proposition and Theorem \ref{T for SL_n R intro thm} implies the following theorem that appeared in the introduction (Theorem \ref{super-exp intro thm}):
\begin{theorem}
Let $n \geq 3$ and let $S$ be a finite generating set of $\SL_n (\mathbb{Z})$ (e.g., $S = \lbrace e_{i,j} (\pm 1) : 1 \leq i, j \leq n, i  \neq j \rbrace$).  Let $\Phi_i : \SL_n (\mathbb{Z}) \rightarrow \SL_n (\mathbb{Z} / i \mathbb{Z})$ be the natural surjective homomorphism.  Then the family of Cayley graphs of $\lbrace (\SL_{n} (\mathbb{Z} / i \mathbb{Z} ),  \Phi_i (S)) \rbrace_{i \in \mathbb{N}}$ is a super-expander family.
\end{theorem}

\begin{proof}
Let $\B$ be some uniformly convex Banach space.  By Theorem \ref{l^2 sum of uc spaces thm},  $\ell^2 (\bigcup_i G/N_i; \B)$ is a uniformly convex Banach space.  Thus by Theorem \ref{T for SL_n R intro thm} and Proposition \ref{prop T imply exp prop},  the family  of Cayley graphs of $\lbrace (\SL_{n} (\mathbb{Z} / i \mathbb{Z} ),  \Phi_i (S)) \rbrace_{i \in \mathbb{N}}$ is a $\B$-expander family.
\end{proof}

It was shown in \cite{warpedcones0, warpedcones1, warpedcones2, warpedcones3} that one can construct super-expanders using warped cones arising from an action of a Banach property (T) group on a compact manifold.  Combining this machinery with our Theorem \ref{T for SL_n R intro thm} also leads to a construction of super-expanders as we will briefly now explain.  Let $(M,d_M)$ be a compact Riemannian manifold  and $\Gamma$ be a finitely generated group with finite symmetric generating set $S$.  Assume that $\Gamma$ acts on $M$ by Lipschitz homeomorphisms.  For $ t>0$, define the $t$-level warped cone denoted $(M, d_\Gamma^t)$ to be the metric space such that $d_\Gamma^t$ is the largest metric satisfying: 
\begin{itemize}
\item $d_\Gamma^t (x,y) \leq t d_M (x,y)$ for every $x,y \in M$.
\item $d_\Gamma^t (x,s.x) \leq 1$ for every $x \in M$ and $s \in S$.
\end{itemize}

\begin{remark}
The metric $d_\Gamma^t$ is dependent on the choice of the generating set,  but this dependence will be irrelevant with respect to our application below,  because of the following fact (see \cite{Roe}): For metrics $d_\Gamma^t$ and $(d_\Gamma^t) '$ that correspond to generating sets $S$ and $S'$,  it holds that $d_\Gamma^t$ and $(d_\Gamma^t)' $ are Lipschitz equivalent.  
\end{remark}

The following theorem is a straight-forward implication of Sawicki's main result in \cite{warpedcones3}:
\begin{theorem}\cite[Theorem 1.1]{warpedcones3}
Let $(M,d_M)$ be a compact  Riemannian manifold and $\Gamma$ be a finitely generated group acting on $M$ by Lipschitz homeomorphisms.  If $\Gamma$ has property $(T_{\B})$ for every super-reflexive Banach space $\B$,  then for every increasing sequence $\lbrace t_i \rbrace_{i \in \mathbb{N}} \subseteq \mathbb{R}_{>0}$ tending to infinity,   the family $\lbrace (M,  d^{t_i}_{\Gamma}) \rbrace_{i \in \mathbb{N}}$ is quasi-isometric to a super-expander.
\end{theorem}

Combining this theorem with Theorem \ref{T for SL_n R intro thm} leads to the following theorem stated in the introduction:

\begin{theorem}
Let $n \geq 3$ and let $(M,d_M)$ be a compact Riemannian manifold such that $\SL_n (\mathbb{Z})$ acts on $M$ by Lipschitz homeomorphisms.  
For every increasing sequence $\lbrace t_i \rbrace_{i \in \mathbb{N}} \subseteq \mathbb{R}$ tending to infinity,   the family $\lbrace (M,  d^{t_i}_{\SL_n (\mathbb{Z})}) \rbrace_{i \in \mathbb{N}}$ is quasi-isometric to a super-expander.
\end{theorem}

\subsection{Property $(FF_\B)$ for $\SL_{n} (\mathbb{R}),  \SL_{n} (\mathbb{Z})$}
\label{prop FFB subsec}
When $\B$ is reflexive it follows for the Ryll-Nardzewski fixed-point Theorem that $G$ has property $(F_{\B})$ if and only if for every isometric linear representation $\pi : G \rightarrow O(\B)$ it holds that every continuous $1$-cocycle into $\pi$ is bounded.  This lead to the stronger notion of property $(FF_{\B})$ defined by  Mimura \cite{Mimura1} as a Banach version of Monod's \cite{Monod} property (TT):  Given a continuous isometric linear representation $\pi : G \rightarrow O(\B)$, a continuous \textit{quasi-$1$-cocycle} into $\pi$ is a continuous map $c: G \rightarrow \B$ such that 
$$\sup_{g,h \in G} \Vert c (gh) - (c (g) + \pi (g) c (h))\Vert < \infty.$$
A group $G$ is said to have property  \textit{property $(FF_{\B})$} if for every continuous isometric linear representation $\pi : G \rightarrow O(\B)$ it holds that every continuous quasi-$1$-cocycle into $\pi$ is bounded. 

The following result of de Laat,  Mimura and de la Salle allows one to deduce property $(FF_{\B})$ from property $(T_{\B})$:
\begin{theorem}
\cite[Section 5]{LMS} 
Let $n \geq 3$ and $\B$ be a super-reflexive Banach space.  For $R = \mathbb{Z},  \mathbb{R}$,  if $\SL_n (R)$ has property $(T_{\B})$, then $\SL_{n+2} (R)$ has property $(FF_{\B})$.
\end{theorem}

Combining this theorem with Theorem \ref{T for SL_n R intro thm} yields the following corollary that appeared in the introduction (Corollary \ref{FF intro coro}):
\begin{corollary}
For every $n \geq 5$ and every super-reflexive Banach space $\B$,  the groups $\SL_n (\mathbb{Z}),  \SL_n (\mathbb{R})$ have property $(FF_{\B})$.
\end{corollary}

\bibliographystyle{alpha}
\bibliography{bib1}
\end{document}